\documentclass[11pt]{amsart}

\usepackage{amsmath, amsthm, amssymb, amsfonts, enumerate}
\usepackage[colorlinks=true,linkcolor=blue,urlcolor=blue]{hyperref}
\usepackage{dsfont}
\usepackage{color}
\usepackage{geometry}
\usepackage{todonotes}

\newtheorem{theorem}{Theorem}[section]
\newtheorem{remark}[theorem]{Remark}
\newtheorem{assumption}[theorem]{Assumption}
\newtheorem{lemma}[theorem]{Lemma}
\newtheorem{proposition}[theorem]{Proposition}

\def \cF{{\mathcal F}}
\def \cL{{\mathcal L}}

\def \E{\mathsf{E}}
\def \EE{\widehat{\mathsf{E}}}
\def \P{\mathsf{P}}
\def \R{\mathbb{R}}
\def \F{\mathbb{F}}

\def\X{\widehat{X}}
\def\S{\widehat{S}}

\def\FF{\widehat{F}}

\title[Singular Stochastic Control for Reflected Diffusions]{On a Class of Singular Stochastic Control Problems for Reflected Diffusions}
\author[Ferrari]{Giorgio Ferrari}
\address{G.~Ferrari: Center for Mathematical Economics (IMW), Bielefeld University, Universit\"atsstrasse 25, 33615, Bielefeld, Germany}
\email{\href{mailto:giorgio.ferrari@uni-bielefeld.de}{giorgio.ferrari@uni-bielefeld.de}}
\date{\today}

\numberwithin{equation}{section}

\begin{document}

\begin{abstract} 
Reflected diffusions naturally arise in many problems from applications ranging from economics and mathematical biology to queueing theory. In this paper we consider a class of infinite time-horizon singular stochastic control problems for a general one-dimensional diffusion that is reflected at zero. We assume that exerting control leads to a state-dependent instantaneous reward, whereas reflecting the diffusion at zero gives rise to a proportional cost with constant marginal value. The aim is to maximize the total expected reward, minus the total expected cost of reflection. We show that depending on the properties of the state-dependent instantaneous reward we can have qualitatively different kinds of optimal strategies. The techniques employed are those of stochastic control and of the theory of linear diffusions.
\end{abstract}

\maketitle

\smallskip


\smallskip

{\textbf{MSC2010 subject classification}}: 93E20, 60J60, 60G40, 91B30, 91B76.

\section{Introduction}

Many problems from applications involve reflected diffusions. These arise, e.g., in economics in the study of currency exchange rate target-zone models pioneered in \cite{Krugman} (see also \cite{Jong} and \cite{Larsen}), or in equilibrium models of competitive firms \cite{Alfarano}; in actuarial science in optimal dividend problems with capital injections (\cite{LokkaZervos}, Chapter 2.5 in \cite{Schmidli} and \cite{Yang}); in queueing theory in heavy-traffic approximations (see, e.g., \cite{Harrison} and \cite{LeeWerasinghe}), and in approximations of queueing systems with reneging or balking \cite{Ward}; in mathematical biology in studies on population dynamics \cite{Ricciardi} (see also the introduction of \cite{RicciardiSacerdote} for further references); in game theory in the problem of constructing best-replies in nonzero-sum games of singular controls \cite{DeAFe16}.

In this paper we consider a class of infinite time-horizon singular stochastic control problems in which the state variable $X^D$ is a general one-dimensional diffusion reflected at zero, and linearly controlled through a nondecreasing c\`agl\`ad process $D$; that is (cf.\ \eqref{state:X} below),
$$dX^{D}_t=\mu(X^{D}_t)dt+\sigma(X^{D}_t)dB_t+dL^D_t-dD_t,\qquad X^{D}_{0}=x\geq 0.$$
Here $\mu$ and $\sigma$ are suitable drift and volatility coefficients, whereas $L^D$ is a local-time-like process that makes the process $X^D$ reflected at zero.

For any given $x\geq 0$, the controller aims at choosing a nondecreasing process $D$ (belonging to a suitable class) in order to solve 
\begin{equation}
\label{probl-intro}
\sup_D\E_x\bigg[\int_0^{\infty}e^{-rt} \eta(X^D_t) \circ dD_t - \kappa \int_0^{\infty}e^{-rt}dL^D_t\bigg].
\end{equation}
That is a so-called \emph{reflected follower} problem (see, e.g., \cite{Baldursson}, \cite{ELKK88}, \cite{ELKK}, and \cite{KS85} as classical references).

In \eqref{probl-intro} $\E_x$ denotes the expectation conditioned on the fact that $X^D_0=x$; $r>0$ is the controller's discount factor; $\eta$ is a state-dependent real-valued function giving the marginal reward accrued by following a control policy $D$, and $\kappa\geq \eta(0)$ is the proportional marginal cost of reflecting the process $X^D$ at zero. The integral with respect to the random measure on $[0,\infty)$ induced by the nondecreasing process $D$ is suitably defined in order to take care of the continuous and jump parts of $D$ (see \eqref{defintegral} below, and \cite{LZ11} and \cite{Zhu92}, among others).

Two problems from applications leading to \eqref{probl-intro} are those of optimal dividends' distribution with capital injections (see, e.g., \cite{KulenkoSchmidli}, \cite{LokkaZervos}, and \cite{Yang}), and of optimal harvesting (see, e.g., \cite{AlvarezShepp}, \cite{Alvarez00}, and references therein). 

In the first case, $X^D$ models the surplus process of a company, $D_t$ represents the cumulative amount of dividends paid to the shareholders up to time $t\geq0$, and $L^D_t$ is the cumulative amount of capital that the shareholder have been asked to inject up to time $t\geq0$ in order to avoid bankruptcy of the company. Taking $\eta(x)=\eta_o$ for all $x\geq0$, $\eta_o$ can be interpreted as the net proportion of leakages from the surplus received by the shareholders after surplus-dependent transaction costs/taxes have been paid, whereas $\kappa$ is a proportional administration cost for capital injections. 

On the other hand, \eqref{probl-intro} might also be seen as an optimal harvesting problem, in which the harvester has the obligation of preserving a population that can extinguish in finite time. In this case, $X^D$ is the size of the population, $D_t$ is the cumulative amount of members of the population that have been harvested up to time $t\geq 0$, and $L^D_t$ is the cumulative amount of members of the species that have been reintroduced up to time $t\geq 0$ in order to avoid extinction. The function $\eta$ now represents a state-dependent marginal yield measuring the instantaneous returns accrued from harvesting, while $\kappa$ is a proportional cost for the reintroduction of the species.

In this paper we solve problem \eqref{probl-intro} under different requirements on the function $\eta$. In particular, assuming (among other more technical conditions) that $\eta \in C^2(\mathbb{R}_+)$, $\mu$ and $\sigma$ are continuously differentiable, and that $r>\mu'$, we conduct our analysis in the three cases: 
\begin{itemize}
\item[\textbf{(A)}] 
$$\big(\cL_{\X}-(r-\mu'(x)\big)\eta(x) < 0 \quad \text{for all $x\geq0$}.$$
\vspace{0.05cm}

\item[\textbf{(B)}] 
There exists $\overline{x} > 0$ such that
\begin{align*}
\big(\cL_{\X}-(r-\mu'(x)\big)\eta(x)
\left\{
\begin{array}{l}
> 0, \quad \text{for $0 < x < \overline{x}$,}\\[+6pt]
< 0, \quad \text{for $x > \overline{x}$}.
\end{array}
\right.
\end{align*}
\vspace{0.05cm}

\item[\textbf{(C)}] 
\begin{align*}
\big(\cL_{\X}-(r-\mu'(x)\big)\eta(x) \geq 0, \qquad \mbox{ for all }\,\,x \geq 0.
\end{align*}
\end{itemize}
\noindent Here $\cL_{\X}$ denotes the second-order differential operator such that $(\cL_{\X} f)\,(x):=\frac{1}{2}\sigma^2(x)f''(x)+(\mu(x)+\sigma(x)\sigma'(x))f'(x)$, $f\in C^2(\R)$, $x\in\R$.

We find that the optimal control is qualitatively different across the three cases. In Case (A), if $\kappa > \eta(0)$ it is optimal to reflect the state process at zero and at an endogenously determined level $b^*$ (a so-called free boundary); if $\kappa = \eta(0)$, an $\varepsilon$-optimal strategy prescribes (apart an initial jump) to keep the state process inside the stripe $[0,b^*_{\delta}]$, where, for any arbitrarily small $\delta>0$, $b^*_{\delta}$ is the free boundary associated to the problem in which $\kappa=\eta(0)+\delta$. Informally, it is optimal to keep the state process at $0$ and to exert control whenever the state process attempts to become strictly positive. In Case (B), it is optimal to reflect the state process at zero and at a free boundary $b^* > \overline{x}$, whereas $D^*\equiv 0$ is the optimal control in Case (C). In the last case the optimal value turns out to be negative. 

In our study we combine probabilistic and analytic techniques, and we follow a \emph{guess-and-verify} approach. Starting from an educated guess on the structure of the optimal control, in each case we find a $C^2$-solution to the Hamilton-Jacobi-Bellman equation that the problem's value function should satisfy. The latter takes the form of a variational inequality with a state-dependent gradient constraint, complemented by a Neumann boundary condition at zero. We then provide a verification theorem in order to confirm the optimality of our initial guess. As a byproduct of the explicit nature of our findings, we are able to perform a sensitivity analysis of the free boundary and of the value function in a case study in which $\eta(x)=\eta_o>0$, and $X^D$ evolves as a linearly controlled reflected Ornstein-Uhlenbeck process (see Section \ref{sec:CS}). Such a case study well models an optimal dividends' distribution problem with compulsory capital injections and with a mean-reverting cash reservoir (see the introduction of \cite{Cadenillas} for a discussion on such a dynamics for the cash balance).

The number of papers dealing with singular stochastic control problems for reflected diffusions is still quite limited. Some early papers have investigated questions related to the existence of an optimal control, or to the validity of the dynamic programming principle for general reflected-follower problems (see \cite{Ma93} and \cite{Min}). The relation to optimal stopping problems with absorption at zero has been studied in \cite{Baldursson}, \cite{ELKK} and \cite{KS85}. However, in none of the previous contributions the explicit form of the optimal control/value function is provided.

An explicit construction of the optimal control and of the value function has been obtained in several papers motivated by optimal dividend problems (see \cite{LokkaZervos}, among others) and by queueing models (see, e.g., \cite{KrichaginaTaksar}). In these works the underlying state process typically obeys a specific diffusive dynamics (e.g., it is given by a Brownian motion with drift). In \cite{Shreveetal} and \cite{Yang} the optimal control and the value function have been obtained for singular stochastic problems involving a general one-dimensional reflected diffusion (the problem addressed in \cite{Yang} is motivated by an optimal dividend problem with capital injections). However, in \cite{Shreveetal} and \cite{Yang} the marginal reward accrued from exerting control is constant, while in our paper it is state-dependent. 

Our paper thus contributes to the literature by providing the explicit form of the optimal control and of the value function in a class of infinite time-horizon, singular stochastic control problems in which: (i) the state process is a general one-dimensional reflected diffusion; (ii) the marginal reward accrued by exerting control is state-dependent; (iii) the reflection at zero is costly. To the best of our knowledge, a similar problem has not been yet solved. Moreover, given the wide range of different fields in which reflected diffusions arise, we believe that our theoretical study might be useful when it comes to solve specific problems motivated by some application.

The rest of the paper is organized as follows. In Section \ref{sec:problem} we set up the problem, by introducing the probabilistic setting (Section \ref{sec:setting}) and the control problem that is the object of our study (Section \ref{sec:probl}). The latter is then solved in Section \ref{sec:solving} in each of the three Cases (A), (B) and (C) previously discussed. In Section \ref{sec:CS} we provide a case study with mean-reverting dynamics and constant reward function. Finally, an appendix is devoted to the proof of some auxiliary result.

%%%%%%%%%%%%%%%%%%%%%%%%%%%%%%%%%%%%%%%%%%%%%

\section{Problem formulation}
\label{sec:problem}

\subsection{The Probabilistic Setting}
\label{sec:setting}

Denote by $(\Omega, \mathcal{F},\P)$ a complete probability space equipped with a filtration $\F=(\mathcal{F}_t)_{t \geq 0}$ under the usual hypotheses. Let $B=(B_t)_{t\geq 0}$ be a one-dimensional standard Brownian motion adapted to $\F$, and $(X_t)_{t\ge 0}$ satisfying the stochastic differential equation (SDE)  
\begin{equation}
\label{state:X1}
dX_t=\mu(X_t)dt+\sigma(X_t)dB_t,\qquad X_0=x\in \mathbb{R}.
\end{equation}
\begin{assumption}
\label{ass:coef}
The functions $\mu: \R \to \R$ and $\sigma: \R \to (0,\infty)$ belong to $C^1(\mathbb{R})$. Moreover, there exists $L>0$ such that
\begin{equation}
\label{eq:Lip}
|\mu(x) - \mu(y)| + |\sigma(x)-\sigma(y)| \leq L|x-y|, \qquad x,y \in \R. 
\end{equation}
%and
%\begin{equation}
%\label{eq:growth}
%|\mu(x)| + |\sigma(x)| \leq M(1 + |x|), \qquad x \in \R. 
%\end{equation}
\end{assumption}
\noindent Continuity of $\mu$ and $\sigma$ together with \eqref{eq:Lip} imply that 
\begin{equation}
\label{eq:growth}
|\mu(x)| + |\sigma(x)| \leq M(1 + |x|), \qquad x \in \R,
\end{equation}
for some $M>0$. Hence, under Assumption \ref{ass:coef}, equation \eqref{state:X1} admits a unique strong solution (cf.\ Chapter 5.5 in \cite{KS}, among others). Moreover, for any $x \in \R$ there exists $\varepsilon_o>0$ such that
\begin{equation}
\label{LI-X}
\int_{x-\varepsilon_o}^{x+\varepsilon_o}\frac{1 + |\mu(z)|}{|\sigma(z)|^2}\,dz < +\infty,
\end{equation}
and therefore the diffusion $X$ is regular; that is, starting from $x\in \R$, $X$ reaches any other point $z\in \R$ in finite time with positive probability.

Given a process $D=(D_t)_{t\ge 0}$ belonging to the set
\begin{align}
\label{def:setS}
& \mathcal{S}:=\{D:\Omega \times \R_+ \mapsto \R_+,\,\,\mathbb{F}-\text{adapted, such that a.s.}\,\,t\mapsto D_t\,\,\text{is}  \\
& \hspace{0.3cm} \qquad \text{nondecreasing, left-continuous},\,\,D_{0}=0\},\nonumber
%\,\,\text{a.s.\ and}\,\,D_{t+}-D_{t} \leq X^D_{t}\,\,\text{for all}\,\,t\geq 0
\end{align}
we then consider the SDE with reflecting boundary condition
\begin{equation}
\label{state:X}
dX^{D}_t=\mu(X^{D}_t)dt+\sigma(X^{D}_t)dB_t+dL^D_t-dD_t,\qquad X^{D}_{0}=x\geq 0.
\end{equation}
Here $L^D \in \mathcal{S}$ is a local-time-like process that assures $X^{D}_t \geq 0$ a.s.\ for all $t\geq 0$, that is flat off $\{t: X^{D,L^D}_t=0\}$, and that has a jump whenever $X^D$ attempts a jump across the origin.  

Under Assumption \ref{ass:coef} it is well known that for any given $D\in \mathcal{S}$ and $x\geq0$ there exists a unique pathwise solution to \eqref{state:X} (see, e.g., \cite{Ma93}, Theorem 4.3, and also \cite{ElKarouiChaleyat}, Theorem 15 in Section 3.2, and Section 3.4). To account for the dependence of $X^D$ on its initial position, from now on we shall write $X^{x,D}$, where appropriate. In the rest of the paper we use the notation $\E_x[ f(X^{D}_t)]=\E[f(X^{x,D}_t)]$, for $f$ Borel-measurable and such that the previous expectations are finite. Here $\E_x$ is the expectation under the measure $\P_x(\,\cdot\,):=\P(\,\cdot\,|X^D_{0}=x)$ on $(\Omega,\cF)$. Furthermore, by \cite{Ma93} we know that $(X^{x,D},L^{x,D})$ solves a so-called discontinuous Skorokhod reflection problem at zero, and since $x\geq0$, for any given $D\in \mathcal{S}$, $L^{x,D}$ has continuous sample paths.

We also consider a diffusion $\X=(\widehat{X}_t)_{t\ge 0}$ evolving according to 
\begin{align}
\label{state:XX}
d\X_t=\big(\mu(\X_t)+\sigma(\X_t)\sigma'(\X_t)\big)dt+\sigma(\X_t)d\widehat{B}_t,\quad \X_0=x\in\R.
\end{align}
Assumption \ref{ass:coef} ensures that the above SDE admits a weak solution $(\widehat{\Omega},\widehat{\mathcal{F}},\widehat{\F},\widehat{\P}_x,\widehat{B},\X)$ which is unique in law up to a possible explosion time (see, e.g., Chapter 5.5 in \cite{KS}). Indeed for every $x \in \R$ there exists $\widehat{\varepsilon}_o>0$ such that
\begin{equation}
\label{LI}
\int_{x-\widehat{\varepsilon}_o}^{x+\widehat{\varepsilon}_o}\frac{1 + |\mu(z)|+|\sigma(z)\sigma'(z)|}{|\sigma(z)|^2}\,dz < +\infty.
\end{equation}
We assume that such a weak solution is fixed for each initial condition $x\in \mathbb{R}$ throughout the paper. Because of \eqref{LI}, the linear diffusion $\X$ is regular as well. In the following we write $\X^x$ to denote the solution of \eqref{state:XX} starting from level $x \in \R$ at time zero. Moreover, we denote by $\widehat{\E}_x$ the expectation under the measure $\widehat{\P}_x$.

\begin{remark}
\label{rem:tilde}
Let us consider the dynamics $X$ of \eqref{state:X1} under the measure $\P_x$, and the dynamics $\X$ under the measure $\widehat{\P}_x$. Let us also define a new measure $\mathsf{Q}_x$ through the Radon-Nikodym derivative 
\begin{align*}
Z_t:=\frac{d\mathsf{Q}_x}{d\P_x}\bigg|_{\cF_t}=\exp\Big\{\int^t_0\sigma'(X_s)dB_s-\frac{1}{2}\int^t_0(\sigma')^2(X_s)ds\Big\},\qquad\P_x-\text{a.s.}
\end{align*}
which is an exponential martingale because $\sigma'$ is bounded by \eqref{eq:Lip}. Hence Girsanov theorem implies that the process $\widehat{B}_t:=B_t-\int_0^t\sigma'(X_s)ds$ is a standard Brownian motion under $\mathsf{Q}_x$, and it is not hard to verify that $\text{Law}\,(X\big|\mathsf{Q}_x)=\text{Law}\,(\X\big|\widehat{\P}_x)$.  
\end{remark}

The infinitesimal generator of the uncontrolled diffusion $X$ is denoted by $\cL_{X}$ and is defined as
\begin{align}
\label{eq:LX}
(\cL_{X} f)\,(x):=\frac{1}{2}\sigma^2(x)f''(x)+\mu(x)f'(x),\quad f\in C^2(\R),\, x\in\R,
\end{align}
whereas the one for $\X$ is denoted by $\cL_{\X}$ and is defined as
\begin{align}
\label{eq:LXhat}
(\cL_{\X} f)\,(x):=\frac{1}{2}\sigma^2(x)f''(x)+(\mu(x)+\sigma(x)\sigma'(x))f'(x),\quad f\in C^2(\R), x\in\R.
\end{align}
Letting $r>0$ be a fixed constant, we make the following standing assumption.

\begin{assumption}
\label{ass:rate}
$\inf_{x\in \mathbb{R}}\big(r-\mu'(x)\big) >0$.
\end{assumption}
\noindent The parameter $r>0$ will be the discount factor in the optimal control problem thas is the object of our study. 

Then we introduce $\psi$ and $\phi$ as the fundamental solutions of the ordinary differential equation (ODE)
\begin{align}
\label{ODE}
\cL_X u(x)-ru(x)=0,\qquad x\in\R,
\end{align}
and we recall that they are strictly increasing and decreasing, respectively (see Ch.~2, Sec.~10 of \cite{BS}). Also, for some arbitrary $x_o \in \mathbb{R}$, we denote by 
$$S'(x):= \exp\Big\{-2\int_{x_o}^x \frac{\mu(y)}{\sigma^2(y)}dy\Big\}, \quad x\in\R,$$ 
the density of the scale function of $(X_t)_{t\ge 0}$, and by $W$ the positive constant Wronskian
\begin{equation}
\label{WronskianX}
W:= \frac{{\psi}'(x){\phi}(x) - {\phi}'(x){\psi}(x)}{S'(x)}, \quad x \in \R.
\end{equation}

Moreover, under Assumptions \ref{ass:coef} and \ref{ass:rate}, any solution of the ODE
\begin{align}
\label{ODE2}
\cL_{\X} u(x)-(r-\mu'(x))u(x)=0,\qquad x\in\R,
\end{align}
can be written as a linear combination of the fundamental solutions $\widehat{\psi}$ and $\widehat{\phi}$. These two functions are strictly increasing and strictly decreasing, respectively. 

Finally, for an arbitrary $x_o\in \mathbb{R}$, we denote by 
$$\S'(x):= \exp\Big\{-2\int_{x_o}^x \frac{\mu(y) + \sigma(y)\sigma'(y)}{\sigma^2(y)}dy\Big\}, \quad x\in\R,$$
the derivative of the scale function of $(\X_t)_{t\ge 0}$, by 
\begin{equation}
\label{hatm}
\widehat{m}'(x):= \frac{2}{\sigma^2(x)\S'(x)}, \quad x \in \mathbb{R},
\end{equation}
the density of the speed measure of $(\X_t)_{t\ge 0}$, and by $w$ the Wronskian
\begin{equation}
\label{Wronskian}
w:= \frac{\widehat{\psi}'(x)\widehat{\phi}(x) - \widehat{\phi}'(x)\widehat{\psi}(x)}{\S'(x)}, \quad x \in \R,
\end{equation}
which is again a positive constant.

\begin{remark}
It is easy to see that the scale functions and of the speed measures of the two diffusions $X$ and $\X$ are related through $\S'(x)=S(x)/\sigma^2(x)$ and $\widehat{m}'(x)=2/S'(x)$ for $x \in \R$.
\end{remark}

Concerning the boundary behavior of the real-valued It\^o-diffusions $X$ and $\X$, in the rest of this paper we assume that $+\infty$ is natural, whereas $-\infty$ is either natural or entrance-not-exit, hence unattainable (see Section 2 in \cite{BS} for a complete discussion of the boundary behavior of linear diffusions). 

%This in particular means that 
%\begin{equation}
%\label{psiphiproperties1}
%\lim_{x \uparrow \infty}\psi(x) = \infty, \quad \lim_{x \uparrow \infty}\phi(x) = 0, \quad \lim_{x \uparrow \infty}\frac{\psi'(x)}{S'(x)} = \infty, \quad \lim_{x \uparrow \infty}\frac{\phi'(x)}{S'(x)} = 0,
%\end{equation}
%\begin{equation}
%\label{psiphiproperties1-BIS}
%\lim_{x \downarrow -\infty}\psi(x) = 0, \quad \lim_{x \downarrow -\infty}\phi(x) = \infty, \quad \lim_{x \downarrow -\infty}\frac{\psi'(x)}{S'(x)} = 0, \quad \lim_{x \downarrow -\infty}\frac{\phi'(x)}{S'(x)} = -\infty,
%\end{equation}
%and
%\begin{equation}
%\label{psiphiproperties2}
%\lim_{x \uparrow \infty}\widehat{\psi}(x) = \infty, \quad \lim_{x \uparrow \infty}\widehat{\phi}(x) = 0, \quad \lim_{x \uparrow \infty}\frac{\widehat{\psi}'(x)}{\S'(x)} = \infty, \quad \lim_{x \uparrow \infty}\frac{\widehat{\phi}'(x)}{\S'(x)} = 0.
%\end{equation}
%
%On the other hand, we require that $-\infty$ is either natural, entrance-not-exit or exit-not-entrance for the diffusions $X$ and $\X$ (see again \cite{BS}). 

The next standing assumption will also hold true throughout the rest of this paper.
\begin{assumption}
\label{ass:psiprimephiprime}
One has $\lim_{x\downarrow -\infty}\phi'(x)=-\infty$ and $\lim_{x\uparrow +\infty}\psi'(x)=\infty$.
\end{assumption}

\noindent Under Assumption \ref{ass:psiprimephiprime} we show in Lemma \ref{lem:AM} in the Appendix that one has $\widehat{\phi}=-\phi'$ and $\widehat{\psi}=\psi'$ (see also the second part of the proof of Lemma 4.3 in \cite{AlvarezMatomaki}).

\begin{remark}
\label{diffusions}
\begin{enumerate}\hspace{10cm}
\item It is worth noticing that our setting covers important classical dynamics, as the cases of a process $X$ given by a drifted Brownian motion (i.e.\ $\mu(x) = \mu \in \R$ and $\sigma(x)=\sigma>0$), and by a mean-reverting process (i.e.\ $\mu(x) = \mu-\theta x$, for some constants $\theta>0$, $\mu\in \R$, and $\sigma(x)=\sigma>0$). 

\item The results of this paper easily generalize to the case of a state space given by $\mathcal{I}:=(\underline{\ell},\overline{\ell})$, $-\infty \leq \underline{\ell} < \overline{\ell} \leq + \infty$, and with $X^D$ being reflected at a point $x_o \in \mathcal{I}$. However, we decided to stick with $\mathcal{I}=\R$ and $x_o=0$ in order to simplify exposition.
\end{enumerate}
\end{remark}

\subsection{The Singular Control Problem}
\label{sec:probl}

We now introduce the singular stochastic control problem that is the object of our study. 

We suppose that exerting a unit of control $D$ at time $t$ leads to a proportional instantaneous reward $\eta(X^D_t)$. On the other hand, we assume that the reflection of $X^D$ at zero is costly, with marginal constant cost $\kappa>0$. 
The total expected reward, net of the total expected costs of the reflection, is therefore
\begin{equation}
\label{eq:functional}
\mathcal{J}_x(D):=\E_x\bigg[\int_{0}^{\infty} e^{-rt}\eta(X^D_t) \circ dD_t - \kappa \int_{0}^\infty e^{-rt} dL^D_t\bigg],
\end{equation}
where $r>0$ is the controller's discount factor. 

Noticing that that any $D \in \mathcal{S}$ can be expressed as the sum of its continuous part and pure jump part, i.e.
\begin{align}
D_t=D^c_t+\sum_{s< t}\Delta D_s,
\end{align}
where $\Delta D_s:=D_{s+}-D_{s}$, we follow \cite{Zhu92} (see also \cite{LZ11}, among many others) and we define the integral with respect to the control $D$ by
\begin{equation}
\label{defintegral}
\int_{0}^{T} e^{-rt}\eta(X^D_t) \circ dD_t := \int_0^{T} e^{-rt}\eta(X^D_t) dD^c_t + \sum_{t < T} e^{-rt} \int_0^{\Delta D_t} \eta(X^D_{t} - z) dz,
\end{equation}
for all $T>0$ and for any $D \in \mathcal{S}$. On the other hand, the integral with respect to the process $L^D$ is intended in the standard Lebesgue-Stieltjes sense.

\begin{remark}
Notice that the choice of a constant penalty factor $\kappa$ is without loss of generality. Indeed, if $\kappa$ were state-dependent, only its value at $0$, $\kappa(0)$, would have played a role in the functional \eqref{eq:functional} since the support of the random measure on $[0,\infty)$ induced by $L^D$ is given by $\{t\geq 0: X^D_t =0\}$.
\end{remark}

For any given initial value of the state $x\geq 0$, the goal is to maximize $\mathcal{J}_x(D)$ among all the controls $D\in \mathcal{S}$ (cf.\ \eqref{def:setS}) such that  
\begin{equation}
\label{intadmiss}
\E_x\bigg[\int_{0}^{\infty} e^{-rt}|\eta(X^D_t)| \circ dD_t + \kappa \int_{0}^\infty e^{-rt} dL^D_t\bigg] < \infty,
\end{equation}
and
\begin{equation}
\label{intadmiss-2}
D_{t+}-D_{t} \leq X^D_{t}\,\,\text{a.s.\ for all}\,\,t\geq 0.
\end{equation}
Requirement \eqref{intadmiss} guarantees that the functional \eqref{eq:functional} is well defined and finite. The condition $D_{t+}-D_{t} \leq X^D_{t}$ for all $t\geq 0$ rules out the possibility of crossing the origin with a single jump. 

Any control process belonging to $\mathcal{S}$ and satisfying \eqref{intadmiss} and \eqref{intadmiss-2} is called \emph{admissible}. For $x\geq 0$ given and fixed, we denote by $\mathcal{A}(x)$ the family of admissible controls. 

The optimal control problem thus reads 
\begin{equation}
\label{eq:value}
V(x):=\sup_{D \in \mathcal{A}(x)}\mathcal{J}_x(D), \qquad x\geq 0,
\end{equation}
and it takes the form of a singular stochastic control problem for a reflected one-dimensional diffusion; i.e., a so-called reflected follower stochastic control problem. An introduction to singular stochastic control problems can be found in \cite{Shreve}, whereas classical references on reflected follower problems are \cite{Baldursson}, \cite{ELKK88}, and \cite{ELKK}.

\begin{remark}
\label{rem:eta}
\begin{enumerate}\hspace{10cm}
\item Our formulation of the performance criterion allows to accommodate also the case in which the reward functional is 
\begin{equation}
\label{functionalrunning}
\E_x\bigg[\int_{0}^{\infty} e^{-rt}\pi(X^D_t) dt - \alpha\int_0^{\infty}e^{-rt} dD_t\bigg],
\end{equation}
for some real-valued continuous function $\pi$, and some $\alpha>0$. To see this define the function $\Pi$ via the ODE
$$(\cL_{X}-r)\Pi(x)=\pi(x),$$ 
with finiteness condition at zero, $\Pi(0+)<+\infty$, and possibly growth condition at infinity as needed. Then for any $D \in \mathcal{S}$ an application of It\^o-Meyer's formula shows that maximizing \eqref{eq:functional} over all the admissible $D$ is equivalent to the maximization of \eqref{functionalrunning} if we set
$$\eta(x):=\Pi'(x) - \alpha, \qquad \kappa:=\Pi'(0).$$ 

\item As discussed in the introduction, problems of optimal harvesting or of optimal dividends' distribution with (compulsory) capital injections might be modeled as \eqref{eq:value}.

More in detail, problem \eqref{eq:value} might be interpreted as an optimal harvesting problem (see, e.g., \cite{AlvarezShepp} and \cite{Alvarez00}), in which the harvester has the obligation of preserving a population that can extinguish in finite time. In this case, $X^D$ is the (logarithm of the) size of the population, $D_t$ is the cumulative amount of members of the population that have been harvested up to time $t\geq 0$, and $L^D_t$ the cumulative amount of members of the species that have been reintroduced up to time $t\geq 0$ in order to avoid extinction. The function $\eta$ represents a state-dependent marginal yield measuring the instantaneous returns accrued from harvesting, while $\kappa$ is a proportional cost for the reintroduction of the species.

In problems of optimal dividends' distribution with (compulsory) capital injections (see, e.g., \cite{KulenkoSchmidli}, \cite{LokkaZervos}, and \cite{Yang}), $D_t$ is the cumulative amount of dividends paid up to time $t\geq0$, $L^D_t$ is the cumulative amount of capital injected up to time $t\geq0$, and $X^D$ is the surplus process. Taking $\eta(x)=\eta_o$ for all $x\geq0$, $\eta_o$ can be interpreted as the net proportion of leakages from the surplus received by the shareholders after surplus-dependent transaction costs/taxes have been paid. On the other hand, $\kappa$ is a proportional administration cost for capital injections.
\end{enumerate}
\end{remark}

The parameter $\kappa$ and the function $\eta:\R_+ \to \R$ satisfy the following assumption which will hold throughout the rest of the paper.
\begin{assumption}
\label{ass:costs}
\begin{itemize}
\item[(i)] The function $\eta$ belongs to $C^2(\R_+)$ and it is such that 
$$\lim_{x \uparrow \infty} \frac{\eta(x)}{\widehat{\psi}(x)}=0,$$
and
$$\E_x\bigg[\int_0^{\infty} e^{-\int_0^t (r - \mu'(\X_s))ds}\, \big|\big(\cL_{\X} -(r-\mu'(\X_s))\big)\eta(\X_s)\big|\,ds\bigg] < \infty.$$
\item[(ii)] One has $\kappa \geq \eta(0)$.
\end{itemize}
\end{assumption}

\begin{remark}
Notice that if $\kappa < \eta(0)$ the value function might be infinite, as it is shown in the following example.

Assume that $X^D_t=x + \mu t + \sigma B_t - D_t + L^D_t$, $x\geq 0$, and with regard to \eqref{eq:functional} take $r=1$, $\eta(x)=\eta_o$ for all $x\geq0$, and $\kappa<\eta_o$. Then, for arbitrary $\beta>0$ consider the admissible control $\widehat{D}_t:=\beta t$, and notice that for such a choice one has $L^{\widehat{D}}_t=\sup_{0\leq u \leq t}(-x - \mu u - \sigma B_u + \beta u) \vee 0$, $L^{\widehat{D}}_0=0$. Since $L^{\widehat{D}}_t \leq \beta t + R_t$, where $R_t := \sup_{0\leq u \leq t}(-x - \mu u - \sigma B_u) \vee 0$, the sub-optimality of $\widehat{D}$ and an integration by parts in the integral with respect to $dL^{\widehat{D}}$ yield for any $x\geq 0$
\begin{eqnarray*}
& \displaystyle V(x) \geq \beta \eta_o\int_0^{\infty} e^{-t}\, dt - \beta \kappa \int_0^{\infty} e^{-t}\,t dt - \kappa\,\E_x\bigg[\int_0^{\infty}e^{-t}\,R_t dt\bigg]\nonumber \\
& \displaystyle = \beta(\eta_o - \kappa) - \kappa\,\E_x\bigg[\int_0^{\infty}e^{-t}\,R_t dt\bigg], \nonumber
\end{eqnarray*} 
and the latter expression diverges if $\beta \uparrow \infty$.

In order to avoid similar pathological situations we have imposed Assumption \ref{ass:costs}-(ii).
\end{remark}

For any $\widehat{\mathbb{F}}$-stopping time $\tau$, in the following we will make use of the convention
$$e^{-\int_0^{\tau} (r - \mu'(\X_s))ds} := \limsup_{t \uparrow \infty} e^{-\int_0^{t} (r - \mu'(\X_s))ds} = 0\quad \text{on} \quad \{\tau=+\infty\},$$
where the last equality is due to Assumption \ref{ass:rate}. Also, for future reference we notice that, setting $\tau_0:=\inf\{t \geq 0: \X_t \leq 0\}$ $\widehat{\P}_x$-a.s., under Assumption \ref{ass:costs}-(i), an application of Dynkin's formula (up to a standard localization argument) to the process $(e^{-\int_0^{t\wedge \tau_0} (r - \mu'(\X_s))ds}\eta(\X_{t \wedge \tau_0}))_{t\geq 0}$ yields for any $x \geq 0$
\begin{equation}
\label{repr-eta}
\eta(x) = \widehat{\E}_x\bigg[\eta(0) e^{-\int_0^{\tau_0} (r - \mu'(\X_s))ds} - \int_0^{\tau_0} e^{-\int_0^t (r - \mu'(\X_s))ds}\, \big(\cL_{\X} -(r-\mu'(\X_s))\big)\eta(\X_s)\, ds\bigg].
\end{equation}

%%%%%%%%%%%%%%%%%%%%%%%%%%%%%%%%%%%%%%%%%%%%%%%%%%%%%%%%%%%%%%%%%%%%%%

\section{Solving the Problem}
\label{sec:solving}

According to the classical theory of singular stochastic control (see \cite{FlemingSoner}, Chapter VIII), we expect that $V$ identifies with a suitable solution to the Hamilton-Jacobi-Bellman (HJB) equation 
\begin{equation}
\label{HJB}
\max\big\{(\cL_X-r)v(x), \eta(x) - v'(x)\big\} =0, \qquad x \in [0,\infty),
\end{equation}
subject to the Neumann boundary condition $v'(0)=\kappa$. The latter is due to the fact that the underlying state process is elastically reflected at zero.

In this section we solve problem \eqref{eq:value} by finding a classical solution to \eqref{HJB} and then verifying its optimality through a verification theorem. This in turn will give us also the optimal control. Our analysis will be developed in each of the following three cases.

\begin{itemize}
\item[\textbf{(A)}] 
$$\big(\cL_{\X}-(r-\mu'(x)\big)\eta(x) < 0 \quad \text{for all $x\geq0$}.$$
\vspace{0.12cm}

\item[\textbf{(B)}] 
There exists $\overline{x} > 0$ such that
\begin{align*}
\big(\cL_{\X}-(r-\mu'(x)\big)\eta(x)
\left\{
\begin{array}{l}
> 0, \quad \text{for $0 < x < \overline{x}$,}\\[+6pt]
< 0, \quad \text{for $x > \overline{x}$}.
\end{array}
\right.
\end{align*}
\vspace{0.12cm}

\item[\textbf{(C)}] 
\begin{align*}
\big(\cL_{\X}-(r-\mu'(x)\big)\eta(x) \geq 0, \qquad \mbox{ for all }\,\,x \geq 0
\end{align*}
\end{itemize}

\begin{remark}
\label{remark:examples}
Notice that under Assumption \ref{ass:rate}, the benchmark case $\eta(x)=\eta_o>0$ clearly fulfills the condition of Case (A), independently of the dynamics $\X$ that we choose. The requirement of Case (B) is commonly assumed in the singular stochastic control literature (see, e.g., Theorem 2 in \cite{Alvarez00} and Assumption 5 in \cite{Jacketal} for similar conditions) in order to ensure that the optimal control is triggered by an optimal reflection boundary. It is easy to see that $\eta(x)=x$ satisfies the condition of Case (B) for $d\X_t = \mu dt + \sigma d\widehat{B}_t$ or $d\X_t = \theta(\mu -\X_t)dt + \sigma d\widehat{B}_t$, killed at rate $r$ and $r+\theta$, respectively. On the other hand, $\eta(x)=e^{-\lambda x}$, $\lambda>0$, and $d\X_t = \mu dt + \sigma d\widehat{B}_t$ satisfy the requirement of Case (A) if $\frac{1}{2}\sigma^2\lambda^2 -\lambda \mu-r<0$, and the condition of Case (C) if $\frac{1}{2}\sigma^2\lambda^2 -\lambda \mu-r>0$.
\end{remark}

In the next sections we will solve problem \eqref{eq:value} as follows:
\begin{itemize}
\item[(i)] in Cases (A) and (B), under the condition that $\kappa>\eta(0)$ (see Section \ref{sec:caseA1});
\item[(ii)] in Cases (A) and (B), under the condition that $\kappa=\eta(0)$ (see Section \ref{sec:caseA2});
\item[(iii)] in Case (C) (see Section \ref{sec:caseC}).
\end{itemize}
We will see that the structure of the optimal control will be qualitatively different across the different problem's configurations. 

%%%%%%%%%%%%%%%%%%%%%%%%%%%%%%%%%%%%%%%%%%%

\subsection{Cases (A) and (B) under the requirement $\kappa>\eta(0)$.}
\label{sec:caseA1}

We guess that in these cases it is optimal to be inactive whenever the state process lies in the interval $[0,b)$, for some $b>0$ to be found, and let the state jump to $b>0$, whenever its current level $x$ is such that $x>b$. By \eqref{defintegral} this leads us to conjecture that $(\cL_X - r)v(x) = 0$ on $[0,b)$ and $v(x)-v(b) = \int_b^x \eta(y)dy$ for any $x> b$. Then supposing that the value function is two-times continuously differentiable at the critical level $b$, the previous conjecture on the optimal policy leads us to reformulate the HJB equation \eqref{HJB} as the free-boundary problem:
\begin{align}
\label{FBP}
\left\{
\begin{array}{l}
(\cL_X - r)v(x) = 0,\quad 0\leq x<b,\\[+6pt]
(\cL_X - r)v(x) \leq  0,\quad x\geq 0,\\[+6pt]
v'(x) = \eta(x), \quad x \geq b,\\[+6pt]
v'(x) \geq  \eta(x), \quad x \geq 0,\\[+6pt]
v'(b) = \eta(b),\\[+6pt]
v''(b) = \eta'(b), \\[+6pt]
v'(0)=\kappa,
\end{array}
\right.
\end{align}
for the two unknowns $v$ and $b$. The fifth and the sixth requirements in \eqref{FBP} are the so-called \emph{smooth-fit} conditions, typical optimality conditions in singular stochastic control problems (see \cite{FlemingSoner}, Chapter VIII).

The general solution to the first equation of \eqref{FBP} is given by
\begin{equation}
v(x) = \alpha \psi(x) + \beta \phi(x), \qquad x \in [0,b),
\end{equation}
where $\psi$ and $\phi$ have been introduced in Section \ref{sec:setting}. Imposing the smooth-fit conditions at $b$ (i.e.\ $v'(b) = \eta(b)$ and $v''(b) = \eta'(b)$) we find after some simple algebra that
\begin{equation}
\label{AB}
\alpha =\frac{\eta(b)\phi''(b) - \phi'(b)\eta'(b)}{\phi''(b)\psi'(b) - \psi''(b)\phi'(b)}, \qquad \beta =\frac{\eta'(b)\psi'(b) - \psi''(b)\eta(b)}{\phi''(b)\psi'(b) - \psi''(b)\phi'(b)}.
\end{equation}
Because under Assumption \ref{ass:psiprimephiprime} we have $\psi'=\widehat{\psi}$ and $\phi'=-\widehat{\phi}$ (see Lemma \ref{lem:AM}) we can rewrite \eqref{AB} as
\begin{equation}
\label{AB2}
\alpha =\frac{\widehat{\phi}(b)\eta'(b) - \eta(b)\widehat{\phi}'(b)}{w\widehat{S}'(b)}, \qquad \beta =\frac{\eta'(b)\widehat{\psi}(b) - \widehat{\psi}'(b)\eta(b)}{w\widehat{S}'(b)},
\end{equation}
upon recalling \eqref{Wronskian}. By setting for any $f\in C^1(\R)$
\begin{equation}
\label{If}
I_f(x):=\frac{f(x)}{\S'(x)}\eta'(x) - \frac{f'(x)}{\S'(x)}\eta(x), \quad x \geq 0,
\end{equation}
we notice that $\alpha =w^{-1}I_{\widehat{\phi}}(b)$ and $\beta =w^{-1}I_{\widehat{\psi}}(b)$, and by imposing the Neumann boundary condition at zero (i.e.\ $v'(0)=\kappa$) we find the equation for $b$
\begin{equation}
\label{eqb}
\kappa = \frac{1}{w}\Big[\widehat{\psi}(0)I_{\widehat{\phi}}(b) - \widehat{\phi}(0)I_{\widehat{\psi}}(b)\Big]. 
\end{equation}

We now rewrite \eqref{eqb} in an equivalent, but more useful, form. We introduce
$$F(x):=\frac{1}{w}\Big[\widehat{\psi}(0)I_{\widehat{\phi}}(x) - \widehat{\phi}(0)I_{\widehat{\psi}}(x)\Big], \qquad x\geq 0,$$ 
and using \eqref{Wronskian} and \eqref{If} it is easily verified that $F(0)=\eta(0)$, so that \eqref{eqb} can be rewritten as
\begin{equation}
\label{eqb2}
\kappa = \eta(0) + \int_0^b F'(y) dy. 
\end{equation}
Here $F'$ is well defined since $\eta \in C^2(\R_+)$ by assumption, as well as $\widehat{\psi}, \widehat{\phi} \in C^2(\R)$.

Notice now that standard differentiation, and the fact that $\cL_{\X}\S=0$ and $(\cL_{\X} - (r-\mu'))g =0$, for $g=\widehat{\psi},\widehat{\phi}$, give 
\begin{equation}
\label{derivIf}
I'_{\widehat{\psi}}(x) = \widehat{m}'(x)\widehat{\psi}(x)(\cL_{\X} - (r-\mu'(x)))\eta(x), \quad I'_{\widehat{\phi}}(x) = \widehat{m}'(x)\widehat{\phi}(x)(\cL_{\X} - (r-\mu'(x)))\eta(x).
\end{equation}
Hence, one has from \eqref{eqb2}
\begin{equation}
\label{eqb3}
\kappa = \eta(0) + \frac{1}{w}\widehat{\phi}(0)\widehat{\psi}(0)\int_0^b \widehat{m}'(y)h(y)\big(\cL_{\X} - (r-\mu'(y))\big)\eta(y) dy,
\end{equation}
where we have set $h(x):=\frac{\widehat{\phi}(x)}{\widehat{\phi}(0)}-\frac{\widehat{\psi}(x)}{\widehat{\psi}(0)}$. Notice that $h(x)<0$ for all $x>0$ since $\widehat{\phi}$ is strictly decreasing and $\widehat{\psi}$ is strictly increasing.

\begin{proposition}
\label{prop:b}
Suppose that Case (A) holds true. Then there exists a unique $b^*>0$ solving \eqref{eqb3}. On the other hand, if Case (B) holds true then there exists a unique $b^*>\overline{x}>0$ solving \eqref{eqb3}.
\end{proposition}
\begin{proof}
Setting $$\Phi(b):= \eta(0)-\kappa + \frac{1}{w}\widehat{\phi}(0)\widehat{\psi}(0)\int_0^b \widehat{m}'(y)h(y)\big(\cL_{\X} - (r-\mu'(y))\big)\eta(y) dy,$$
it is clear that \eqref{eqb3} is equivalent to $\Phi(b)=0$.

First of all we notice that $\Phi(0) = \eta(0)-\kappa <0$. Moreover, if Case (B) holds true then $\Phi(\overline{x}) < 0$ because $(\cL_{\X} - (r-\mu'(y)))\eta(y) > 0$ for any $0<y < \overline{x}$, $h(y) < 0$ for all $y>0$, and $\eta(0)-\kappa <0$.

Also, we have
$$\Phi'(b) = \frac{1}{w} \widehat{\phi}(0)\widehat{\psi}(0)\widehat{m}'(b)h(b)\big(\cL_{\X} - (r-\mu'(b))\big)\eta(b),$$
and the following holds: if Case (A) applies then $\Phi'(b) >0$ for all $b>0$ because $(\cL_{\X} - (r-\mu'(b)))\eta(b) <0$ for any $b> 0$ and $h(b) < 0$ for all $b>0$. On the other hand, if Case (B) holds true then $\Phi'(b) >0$ for all $b> \overline{x}$ because $(\cL_{\X} - (r-\mu'(b)))\eta(b) <0$ for any $b> \overline{x}$ and $h(b) < 0$ for all $b>0$.

We now show that 
\begin{equation}
\label{limitPhib}
\lim_{b\uparrow \infty}\Phi(b)=\infty.
\end{equation}
This allows us to conclude that under Case (A) there exists a unique $b^* > 0$ solving $\Phi(b)=0$, whereas under Case (B) there exists a unique $b^* > \overline{x}$ solving $\Phi(b)=0$.

We prove \eqref{limitPhib} under the assumption that Case (B) holds true. Completely analogous arguments can be used to show \eqref{limitPhib} under the condition of Case (A). 
We take an arbitrary $b > \overline{x}$, and we notice that an application of the integral mean value theorem yields for some $\xi \in (\overline{x}, b)$
\begin{align}
& \Phi(b) = \Phi(\overline{x}) + \frac{1}{w}\widehat{\phi}(0)\widehat{\psi}(0)\big(\cL_{\X} - (r-\mu'(\xi))\big)\eta(\xi) \times \nonumber \\
& \times \bigg[\frac{1}{\widehat{\phi}(0)}\int_{\overline{x}}^b \widehat{m}(y)\widehat{\phi}(y)\frac{r-\mu'(y)}{r-\mu'(y)}dy - \frac{1}{\widehat{\psi}(0)}\int_{\overline{x}}^b \widehat{m}(y)\widehat{\psi}(y)\frac{r-\mu'(y)}{r-\mu'(y)}dy  \bigg] \nonumber \\
& \geq \Phi(\overline{x}) + \frac{1}{w}\widehat{\phi}(0)\widehat{\psi}(0)\big(\cL_{\X} - (r-\mu'(\xi))\big)\eta(\xi) \times \\
& \times \left[\frac{1}{r_o\widehat{\phi}(0)}\left(\frac{\widehat{\phi}'(b)}{\S'(b)} - \frac{\widehat{\phi}'(\overline{x})}{\S'(\overline{x})}\right) -
 \frac{1}{(r+L)\widehat{\psi}(0)}\left(\frac{\widehat{\psi}'(b)}{\S'(b)} - \frac{\widehat{\psi}'(\overline{x})}{\S'(\overline{x})}\right)\right]. \nonumber
\end{align}
In the second step above we have used that (cf.\ par.~9 and 10, Ch.~2 of \cite{BS})
\begin{equation*}
%\label{psiphiproperties3b}
\frac{\widehat{\psi}'(\beta)}{\S'(\beta)} - \frac{\widehat{\psi}'(\alpha)}{\S'(\alpha)}= \int_{\alpha}^{\beta}(r-\mu'(y))\widehat{\psi}(y)\widehat{m}'(y) dy, \qquad \frac{\widehat{\phi}'(\beta)}{\S'(\beta)} - \frac{\widehat{\phi}'(\alpha)}{\S'(\alpha)}= \int_{\alpha}^{\beta}(r-\mu'(y))\widehat{\phi}(y)\widehat{m}'(y) dy,
\end{equation*}
for any $-\infty<\alpha<\beta<\infty$, as well as that $(\cL_{\X} - (r-\mu'(\xi))\big)\eta(\xi) <0$, and $\mu'(y) \geq -L$ for all $y\in \R$ (cf.\ \eqref{eq:Lip}). Moreover, we have set $r_o:=\inf_{x}(r-\mu'(x))>0$ by Assumption \ref{ass:rate}.

By recalling that $+\infty$ is natural for $\X$, and therefore that 
$$\lim_{b\uparrow \infty}\frac{\widehat{\phi}'(b)}{\S'(b)} = 0 \quad \text{and} \quad \lim_{b\uparrow \infty}\frac{\widehat{\psi}'(b)}{\S'(b)} = \infty,$$ 
and by using again that $(\cL_{\X} - (r-\mu'(\xi)))\eta(\xi) < 0$ since $\xi \in (\overline{x}, b)$, we can let $b \uparrow \infty$ to conclude that \eqref{limitPhib} holds true. The proof is then completed.
\end{proof}

Given $b^*$ as in Proposition \ref{prop:b}, for any $x\geq 0$ we then set 
\begin{align}
\label{candidate}
v(x)=
\left\{
\begin{array}{l}
\displaystyle \alpha \psi(x) + \beta \phi(x),\quad 0 \leq x \leq b^*,\\[+6pt]
\displaystyle \int_{b^*}^x \eta(y) dy + v(b^*), \quad x \geq b^*
\end{array}
\right.
\end{align}
with $\alpha$ and $\beta$ given by \eqref{AB} (equivalently, by \eqref{AB2}), and, by continuity, $v(b^*)=\alpha\psi(b^*) + \beta\phi(b^*)$. The following proposition shows that $v$ actually solves \eqref{HJB} with the boundary condition $v'(0)=\kappa$.
\begin{proposition}
\label{prop:HJB}
The function $v$ of \eqref{candidate} is a classical solution to the HJB equation \eqref{HJB} with the boundary condition $v'(0)=\kappa$.
\end{proposition}
\begin{proof}
The proof is organized in three steps.\vspace{0.25cm}

\emph{Step 1.} By construction one clearly has that $v \in C^2(\R_+)$, $v'(0)=\kappa > \eta(0)$, $(\cL_X - r)v(x)=0$ on $[0,b^*)$ and $v'(x) = \eta(x)$ on $[b^*,\infty)$.
\vspace{0.25cm}

\emph{Step 2.} Here we show that $(\cL_X - r)v(x)\leq 0$ on $[b^*,\infty)$, so that $(\cL_X - r)v(x)\leq 0$ on $\R_+$.
Then take $x \geq b^*$ given and fixed, recall \eqref{eq:LXhat}, and notice that integrating by parts one obtains
\begin{align}
\label{stimaineq}
& 0 \geq \int_{b^*}^x \big(\cL_{\X} - (r-\mu'(y))\big)\eta(y) dy = \frac{1}{2}\sigma^2(x)\eta'(x) + \mu(x)\eta(x) \nonumber \\
& - r\bigg[\int_{b^*}^x \eta(y) dy + \frac{1}{r}\Big(\frac{1}{2}\sigma^2(b^*)\eta'(b^*) + \mu(b^*)\eta(b^*)\Big)\bigg],
\end{align}
where the first inequality is due to the fact that $b^*>0$ in Case (A) and $b^* > \overline{x}$ in Case (B).
We now claim (and prove later) that
\begin{equation}
\label{vb}
\frac{1}{2}\sigma^2(b^*)\eta'(b^*) + \mu(b^*)\eta(b^*) = r\Big(\alpha\psi(b^*) + \beta\phi(b^*)\Big) = rv(b^*),
\end{equation}
so that we can continue from \eqref{stimaineq} by writing
\begin{align}
\label{stimaineq2}
& 0 \geq \int_{b^*}^x \big(\cL_{\X} - (r-\mu'(y))\big)\eta(y) dy \nonumber \\ 
&= \frac{1}{2}\sigma^2(x)\eta'(x) + \mu(x)\eta(x) - r\bigg[\int_{b^*}^x \eta(y) dy + v(b^*)\bigg]  \\
& = \big(\cL_X - r\big)\bigg(\int_{b^*}^x \eta(y) dy + v(b^*)\bigg) = \big(\cL_X - r\big)v(x), \nonumber
\end{align}
where \eqref{candidate} has been used in the last step.

To complete this part of the proof, it thus just remains to prove \eqref{vb}.
By continuity of $v(\,\cdot\,)$ we have $v(b^*) = \alpha\psi(b^*) + \beta\phi(b^*)$, with $\alpha$ and $\beta$ as in \eqref{AB}; that is
\begin{align}
\label{stimavb}
& v(b^*) = \frac{1}{\phi''(b^*)\psi'(b^*) - \psi''(b^*)\phi'(b^*)}\times \nonumber \\
& \times \Big[\eta(b^*)\Big(\phi''(b^*)\psi(b^*) - \psi''(b^*)\phi(b^*)\Big) + \eta'(b^*)\Big(\psi'(b^*)\phi(b^*) - \phi'(b^*)\psi(b^*)\Big)\Big]. 
\end{align}
By noticing that 
\begin{equation*}
\psi'(b^*)\phi(b^*) - \phi'(b^*)\psi(b^*) = WS'(b^*) \qquad \phi''(b^*)\psi(b^*) - \psi''(b^*)\phi(b^*) = -WS''(b^*),
\end{equation*}
and that $S''(x) = - \frac{2\mu(x)}{\sigma^2(x)}S'(x)$ for any $x \in \R$, we obtain from \eqref{stimavb} that
\begin{equation}
\label{stimavb2}
v(b^*) = \frac{1}{\phi''(b^*)\psi'(b^*) - \psi''(b^*)\phi'(b^*)}\left[\frac{2S'(b^*)}{\sigma^2(b^*)}\right]\left[\frac{1}{2}\sigma^2(b^*)\eta'(b^*) + \mu(b^*)\eta(b^*)\right]. 
\end{equation}
Since now $\phi$ and $\psi$ solve $(\cL_X - r)u=0$ we have for any $x \in \R$
$$\phi''(x) = -\frac{2\mu(x)}{\sigma^2(x)}\phi'(x) + \frac{2r}{\sigma^2(x)}\phi(x), \qquad \psi''(x) = -\frac{2\mu(x)}{\sigma^2(x)}\psi'(x) + \frac{2r}{\sigma^2(x)}\psi(x),$$
and therefore 
$$\phi''(b^*)\psi'(b^*) - \psi''(b^*)\phi'(b^*) = \frac{2r}{\sigma^2(b^*)}WS'(b^*).$$
This last relation used in \eqref{stimavb2} finally yields \eqref{vb}.
\vspace{0.25cm}

\emph{Step 3.} We now show that $v'(x) \geq \eta(x)$ for any $x \in [0,b^*)$, so that $v'(x) \geq \eta(x)$ for all $x \geq 0$. By the Neumann boundary condition we have $v'(0) = \kappa > \eta(0)$, therefore we need to prove the inequality only in the open interval $(0,b^*)$. Notice that by differentiating the first equation of \eqref{FBP}, we find that $v'$ of \eqref{candidate} solves 
$$\big(\cL_{\X} - (r-\mu'(x)\big)v'(x) = 0, \qquad x \in (0,b^*),$$
together with the boundary conditions $v'(0) = \kappa$ and $v'(b^*)=\eta(b^*)$. Recalling $\tau_0=\inf\{t\geq 0: \X^x_t \leq 0\}$, and introducing $\tau_{b^*}:=\inf\{t\geq 0: \X^x_t \geq b^*\}$, an application of Feynman-Kac formula gives
\begin{equation}
\label{reprv}
v'(x) = \EE_x\Big[ \kappa e^{-\int_0^{\tau_0} (r - \mu'(\X_s))ds}\mathds{1}_{\{\tau_0 < \tau_{b^*}\}} + \eta(b^*)e^{-\int_0^{\tau_{b^*}} (r - \mu'(\X_s))ds}\mathds{1}_{\{\tau_0 > \tau_{b^*}\}}\Big].
\end{equation} 

By defining the strictly positive and strictly increasing function $\FF(x):={\widehat{\psi}(x)}/{\widehat{\phi}(x)}$, $x \in \R$, we can use Lemma 2.3 in \cite{Dayanik08} to rewrite \eqref{reprv} in the equivalent form
\begin{equation}
\label{reprv2}
\frac{v'(x)}{\widehat{\phi}(x)} = \frac{\kappa}{\widehat{\phi}(0)}\left[\frac{\FF(b^*) - \FF(x)}{\FF(b^*) - \FF(0)}\right] + \frac{\eta(b^*)}{\widehat{\phi}(b^*)}\left[\frac{\FF(x) - \FF(0)}{\FF(b^*) - \FF(0)}\right].
\end{equation} 

In the spirit of  \cite{Dayanik08} (see also \cite{DayanikKaratzas}) we now make a change of variable. We define $y:=\FF(x)$, and we set $y_o:=\FF(0)$, $y_*:=\FF(b^*)$, and $\overline{y}:=\FF(\overline{x})$ (if Case (B) holds true). Also, for any continuously differentiable real-valued function $f$ we set
$$\widetilde{f}(y):=\Big(\Big(\frac{f}{\widehat{\phi}}\Big) \circ \FF^{-1}\Big)(y), \qquad y \geq y_o.$$

By using these definitions, we can rewrite \eqref{reprv2} in the new scale as $\widetilde{v'}(y) = \vartheta(y)$, $y \geq y_o$, where we have set
\begin{equation}
\label{retta}
\vartheta(y):=\kappa\left[\frac{y_* - y}{y_* - y_o}\right] + \widetilde{\eta}(y_*)\left[\frac{y - y_o}{y_* - y_o}\right], \quad y \geq y_o.
\end{equation}

Notice that $\vartheta$ is the straight line connecting the points $(y_o,\kappa)$ and $(y_*,\widetilde{\eta}(y_*))$. Indeed, by definition, $\vartheta(y_o)=\kappa$ and $\vartheta(y_*)=\widetilde{\eta}(y_*)$. Moreover, $\vartheta$ is tangent to $\widetilde{\eta}$ at the point $(y_*,\widetilde{\eta}(y_*))$. To see this we first evaluate
$$\vartheta'(y_*) =  -\frac{\kappa}{y_* - y_o} + \frac{\widetilde{\eta}(y_*)}{y_* - y_o}.$$ Then, by direct calculations employing the definition of $\widetilde{\eta}$ and the fact that $\widetilde{\eta}'(y_*) = (\frac{\eta}{\widehat{\phi}})'(y_*)/\FF'(y_*)$, we find the following equivalences
\begin{align*}
&\vartheta'(y_*)=\widetilde{\eta}'(y_*) \Longleftrightarrow -\frac{\kappa}{y_* - y_o} + \frac{\widetilde{\eta}(y_*)}{y_* - y_o} =  \widetilde{\eta}'(y_*)  \nonumber \\
& \Longleftrightarrow \Big(\frac{v'}{\widehat{\phi}}\Big)'(b^*) = \Big(\frac{\eta}{\widehat{\phi}}\Big)'(b^*) \Longleftrightarrow \left[\frac{v'(x) - \eta(x)}{\widehat{\phi}(x)}\right]' = 0.
\end{align*}
Since the equality in the last term above holds true due to the fact that $v''(b^*) = \eta'(b^*)$ and $v'(b^*) = \eta(b^*)$, we conclude that $\vartheta$ is tangent to $\widetilde{\eta}$ at the point $(y_*,\widetilde{\eta}(y_*))$.

Recall now that we have either $\big(\cL_{\X} - (r-\mu'(x))\big)\eta(x) < 0$ for all $x\geq 0$ under Case (A), or $\big(\cL_{\X} - (r-\mu'(x))\big)\eta(x) > 0$ for $0\leq x \leq \overline{x}$, and $\big(\cL_{\X} - (r-\mu'(x))\big)\eta(x) < 0$ for $x > \overline{x}$ if Case (B) holds true. 
By Lemma \ref{Daya} in the Appendix we know that the previous relations are equivalent to the fact that $\widetilde{\eta}$ either is strictly concave for any $y\geq y_o$ (if Case (A) holds), or it is strictly convex on $[y_o, \overline{y}]$ and then strictly concave for $y > \overline{y}$ (if Case (B) holds). Together with $\vartheta(y_o) > \widetilde{\eta}(y_o)$ (i.e.\ $\kappa > {\eta}(0)$), these concavity/convexity properties of $\widetilde{\eta}$ give that $\vartheta(y) \geq \widetilde{\eta}(y)$ on $[y_o,y_*]$. Therefore, $v'(x) \geq \eta(x)$ on $[0,b^*]$, and this completes the proof.

\end{proof}

\subsubsection{The Optimal Solution}

Given $x\geq 0$ and $(D,L)\in\mathcal{S}\times\mathcal{S}$, let $Z^{x,D,L}$ be the unique strong solution (see, e.g., \cite{Protter}, Theorem V.7) to
$$dZ^{x,D,L}_t=\mu(Z^{x,D,L}_t)dt+\sigma(Z^{x,D,L}_t)dB_t+dL_t-dD_t,\quad t>0, \qquad Z^{x,D,L}_{0}=x.$$
Then, let $(D^*, L^{*})=(D^*_t, L^{*}_t)_t$ be the couple of nondecreasing processes that solves the double Skorokhod reflection problem $\textbf{SP}(0,b^*;x)$ defined as follows:
\begin{align}
\text{Find $(D,L)\in\mathcal{S}\times\mathcal{S}$ s.t.}
\left\{
\begin{array}{l}
\displaystyle Z^{x,D,L}_t\in[0,b^*], \text{$\P$-a.s.~for $t > 0$},\\[+6pt]
\displaystyle \int^{T}_0{\mathds{1}_{\{Z^{x,D,L}_t>0\}}dL_t}=0, \text{$\P$-a.s.~for any $T>0$,}\\[+6pt]
\displaystyle \int^{T}_0{\mathds{1}_{\{Z^{x,D,L}_t<b^*\}}dD_t}=0, \text{$\P$-a.s.~for any $T>0$.}
\end{array}
\right.
\end{align}
Under Assumption \ref{ass:coef}, Problem $\textbf{SP}(0,b^*;x)$ admits a unique pathwise solution (cf., e.g., Theorem 4.1 in \cite{Tanaka}), which is clearly such that $\text{supp}\{dD^*_t\}\cap \text{supp}\{dL^{*}_t\}=\emptyset$. In the following, we set $D^*_{0}=0=L^*_{0}$. Moreover, given the fact that $x \geq 0$, the process $t \mapsto L^*_t$ is continuous for $t \geq 0$, whereas $D^*$ has continuous sample paths, apart a possible initial jump of amplitude $(x -b^*)$ at time zero. 

Recalling \eqref{state:X}, for the subsequent analysis it is worth bearing in mind that pathwise uniqueness implies $Z^{D^*,L^*} = X^{D^*}$ and $L^*=L^{D^*}$.

\begin{proposition}
Let $x\geq0$. The process $D^*$ such that $(D^*, L^{*})$ solves $\textbf{SP}(0,b^*;x)$ is an admissible control.
\end{proposition}
\begin{proof}
Recall that $Z^{D^*,L^*} = X^{D^*}$ and $L^*=L^{D^*}$, by pathwise uniqueness. Clearly $D^* \in \mathcal{S}$. Also, $D^*_{0+}-D^*_0=(x-b^*)^+\leq x = X^{x,D^*}_0$, whereas $X^{x,D^*}_t\geq 0 = D^*_{t+}-D^*_t$ for any $t> 0$. To prove the admissibility of $D^*$, it thus remains to show that (cf.\ \eqref{intadmiss})
\begin{equation}
\label{integraladm}
\E_x\bigg[\int_{0}^{\infty} e^{-rt}|\eta(X^{D^*}_t)| \circ dD^*_t + \kappa \int_{0}^\infty e^{-rt} dL^{*}_t\bigg] < \infty.
\end{equation}
By \eqref{defintegral} and the fact that $(D^*, L^{*})$ solves $\textbf{SP}(0,b^*;x)$ we have
\begin{align}
\label{integraladm-2}
&\E_x\bigg[\int_{0}^{\infty} e^{-rt}|\eta(X^{D^*}_t)| \circ dD^*_t \bigg] = \int_0^{(x-b^*)^+}|\eta(x - z)| dz \nonumber \\
&+ |\eta(b^*)|\E_x\bigg[\int_{0}^{\infty} e^{-rt} dD^{c,*}_t \bigg],
\end{align}
where we have used that $\text{supp}\{dD^{c,*}\} = \{b^*\}$. By the continuity of $\eta$ we have 
\begin{equation}
\label{inetgral1}
\int_0^{(x-b^*)^+}|\eta(x - z)| dz \leq (x-b^*)^+ \max_{u \in [0,(x-b^*)^+]}|\eta(x-u)| < \infty.
\end{equation}
Also, arguing as in the proof of Lemma 2.1 of \cite{Shreveetal} (see in particular equations (2.16) and (2.17) therein), we have that 
\begin{equation}
\label{inetgral2}
\E_x\bigg[\int_{0}^{\infty} e^{-rt} dD^{c,*}_t \bigg] < \infty \qquad \text{and} \qquad \E_x\bigg[\int_{0}^{\infty} e^{-rt} dL^{*}_t \bigg] < \infty.
\end{equation}
By combining \eqref{inetgral1}, \eqref{inetgral2} and \eqref{integraladm-2}, we conclude that \eqref{integraladm} holds true, and therefore that $D^*\in \mathcal{A}(x)$.
\end{proof}

\begin{theorem}
\label{teo:verify}
Let $(D^*, L^{*})=(D^*_t, L^{*}_t)_t$ be the couple of nondecreasing processes solving $\textbf{SP}(0,b^*;x)$, and let $v$ as in \eqref{candidate}. Then one has that $v=V$ on $\R_+$ and $D^*$ is optimal.
\end{theorem}
\begin{proof}
\emph{Step 1.} Let $x \geq 0$, $D \in \mathcal{A}(x)$, and $(X^{x,D}, L^{x,D})$ the solution of the SDE with reflecting boundary condition at zero \eqref{state:X}. For $n \geq 1$ set $\tau_n:=\inf\{t \geq 0: X^{D}_t \geq n\}$ $\P_x$-a.s. Since $v \in C^2(\R_+)$, we can apply It\^o-Meyer's formula for semimartingales to the process $e^{-rt}v(X^{D}_t)$ on $[0,\tau_n]$ and obtain
\begin{align}
\label{verifico1}
& \E_x\Big[e^{-r\tau_n}v(X^{D}_{\tau_n})\Big] - v(x) = \E_x\bigg[\int_{0}^{\tau_n} e^{-rs}\big(\cL_X -r \big)v(X^{D}_s) ds\bigg] \nonumber \\
& - \E_x\bigg[\int_{0}^{\tau_n} e^{-rs}v'(X^{D}_s) dD^c_s + \sum_{s < \tau_n} e^{-rs}\big(v(X^{D}_{s+}) - v(X^{D}_{s})\big)\bigg] \\
& + \E_x\bigg[\int_{0}^{\tau_n} e^{-rs}v'(X^{D}_s) dL^D_s\bigg]. \nonumber
\end{align} 
Notice that here we have used that for any $D\in \mathcal{A}(x)$, $x\geq0$, $t \mapsto L^D_{t}$ is continuous a.s., and therefore the (random) measure on $[0,\infty)$ $dL^D_{\cdot}$ has only continuous part. 
Because
$$\sum_{s < \tau_n} e^{-rs}\big(v(X^{D}_{s+}) - v(X^{D}_{s})\big) = \sum_{s < \tau_n} e^{-rs}\int_0^{\Delta D_s}v'(X^{D}_s-z) dz,$$
and because $v$ solves \eqref{HJB}, we can continue from \eqref{verifico1} by writing
\begin{align}
\label{verifico2}
& \E_x\Big[e^{-r\tau_n}v(X^{D}_{\tau_n})\Big] - v(x) \leq -\E_x\bigg[\int_{0}^{\tau_n} e^{-rs}v'(X^{D}_s)dD^c_s\bigg] \nonumber \\
& - \E_x\Big[\sum_{s < \tau_n} e^{-rs}\int_{0}^{\Delta D_s}v'(X^{D}_{s}-z)dz\Big] + \E_x\bigg[\int_{0}^{\tau_n} e^{-rs}v'(X^{D}_s)dL^D_s\bigg] \\
& \leq \E_x\bigg[-\int_{0}^{\tau_n} e^{-rs}\eta(X^{D}_s) \circ dD_s + \int_{0}^{\tau_n} e^{-rs}v'(X^{D}_s)dL^D_s\bigg]; \nonumber
\end{align}
that is,
\begin{align}
\label{verifico3}
& v(x) \geq \E_x\Big[e^{-r\tau_n}v(X^{D}_{\tau_n})\Big]  \nonumber \\
& +\E_x\bigg[\int_{0}^{\tau_n} e^{-rs}\eta(X^{D}_s) \circ dD_s - \kappa \int_{0}^{\tau_n} e^{-rs}dL^D_s\bigg].
\end{align}
By admissibility of $D$ we have
\begin{equation}
\label{admissibility}
\E_x\bigg[\int_{0}^{\infty} e^{-rs}|\eta(X^{D}_s)| \circ dD_s\bigg] + \E_x\bigg[\int_{0}^{\infty} e^{-rs}dL^D_s\bigg] < \infty.
\end{equation}
Moreover, by Lemma \ref{lem:trasvers} in Appendix, we also have 
\begin{equation}
\label{trasversality}
\lim_{n\uparrow \infty}\E_x\big[e^{-r\tau_n}v(X^{D}_{\tau_n})\big] = 0.
\end{equation}
Then, noticing that $\tau_n \uparrow \infty$ $\P_x$-a.s.\ for $n \uparrow \infty$, taking limits in \eqref{verifico3}, exploiting \eqref{admissibility} and \eqref{trasversality} we obtain by the dominated convergence theorem that
\begin{align}
\label{verifico4}
& v(x) \geq \E_x\bigg[\int_{0}^{\infty} e^{-rs}\eta(X^{D}_s) \circ dD_s - \kappa \int_{0}^{\infty} e^{-rs}dL^D_s\bigg].
\end{align}
Since the previous holds for any $D \in \mathcal{A}(x)$ and any $x\geq0$, we conclude that $v \geq V$ on $\R_+$.
\vspace{0.25cm}

\emph{Step 2.} Fix again an arbitrary $x\geq 0$, take now $(D^*, L^{*})$ solving $\textbf{SP}(0,b^*;x)$ and recall that $Z^{D^*,L^*} = X^{D^*}$ and $L^*=L^{D^*}$, by pathwise uniqueness. Since $X^{D^*}_t \in [0,b^*]$ $\P_x$-a.s.\ for all $t\geq 0$, all the inequalities leading to \eqref{verifico3} become equalities and thus yield
\begin{align}
\label{verifico3bis}
& v(x) = \E_x\Big[e^{-r\tau_n}v(X^{D^*}_{\tau_n})\Big]  \nonumber \\
& + \E_x\bigg[\int_{0}^{\tau_n} e^{-rs}\eta(X^{D^*}_s) \circ dD^*_s - \kappa \int_{0}^{\tau_n} e^{-rs}dL^{*}_s\bigg].
\end{align}
Recall now that \eqref{integraladm} holds true, and notice that $\lim_{n\uparrow \infty}\E_x\big[e^{-r\tau_n}v(X^{D^{*}}_{\tau_n})\big] = 0$ since $X^{D^*}_t \in [0,b^*]$ $\P_x$-a.s.\ for all $t\geq 0$, $v$ is continuous, and $\tau_n \uparrow \infty$ $\P_x$-a.s.\ as $n \uparrow \infty$. Therefore, by taking limits in \eqref{verifico3bis} as $n\uparrow \infty$, and invoking the dominated convergence theorem we find
\begin{align}
\label{verifico4bis}
& v(x) = \E_x\bigg[\int_{0}^{\infty} e^{-rs}\eta(X^{D^*}_s) \circ dD^*_s - \kappa \int_{0}^{\infty} e^{-rs}dL^{*}_s\bigg] \leq V(x),
\end{align}
where the last equality follows by admissibility of $D^*$. We thus conclude that $v=V$ on $\R_+$ and that $D^*$ is optimal.
\end{proof}

As a byproduct of Theorem \ref{teo:verify} we have the following proposition.

\begin{proposition}
\label{prop:Vprime}
Let $x\geq 0$ and recall $\tau_0=\inf\{t\geq 0: \X_t \leq 0\}$, $\widehat{\P}_x$-a.s. Then one has that
\begin{equation}
\label{vprimeOS}
V'(x) = \sup_{\tau\geq 0}\EE_x\Big[ \kappa e^{-\int_0^{\tau_0} (r - \mu'(\X_s))ds}\mathds{1}_{\{\tau_0 \leq \tau\}} + \eta(\X_{\tau})e^{-\int_0^{\tau} (r - \mu'(\X_s))ds}\mathds{1}_{\{\tau_0 > \tau\}}\Big], 
\end{equation}
and the stopping time $\tau^{\star}:=\inf\{t\geq 0: \X^x_t \geq b^*\}$ is optimal.
\end{proposition}

\begin{proof}
We only sketch the proof as its arguments are standard. Since $V$ equals $v$ of \eqref{candidate} by Theorem \ref{teo:verify}, and $v\in C^2(\mathbb{R}_+)$ by construction, $V'$ belongs to $C^1(\R_+)$. Moreover, it is easy to check from \eqref{candidate} that $V\in C^2(\mathbb{R}_+ \setminus b^*)$.
Also,
$$\big(\cL_{\X} - (r-\mu'(x)\big)V'(x) =0 \quad \mbox{and} \quad V'(x) \geq \eta(x) \quad \mbox{on} \quad [0,b^*),$$
$$V'(x) = \eta(x) \quad \mbox{and} \quad \big(\cL_{\X} - (r-\mu'(x)\big)V'(x) \leq  0 \quad \mbox{on}\,\,[b^*, +\infty),$$
and $V'(0) = \kappa$.

Then, for any $\widehat{\F}$-stopping time $\tau$, an application of It\^o's lemma (through a mollification argument, see e.g.\ Theorem 2.2.1 in \cite{OksendalSulem}) to the process $(e^{-\int_0^{t \wedge \tau_0} (r - \mu'(\X_s))ds}V'(\X^x_{t \wedge \tau_0}))_{t\geq0}$ on the time interval $[0,\tau]$, together with a standard verification argument (see, e.g., Theorem 10.4.1 in \cite{Oksendal}), imply that $V'$ admits the probabilistic representation \eqref{vprimeOS}, and that $\tau^{\star}:=\inf\{t\geq 0: \X^x_t \geq b^*\}$ is optimal.
\end{proof}

\noindent The previous result is consistent with the findings of \cite{Baldursson}, \cite{ELKK} and \cite{KS85} where, for a state process given by a Brownian motion, it is established the connection between reflected follower singular stochastic control problems and optimal stopping problems with absorption at zero.

%%%%%%%%%%%%%%%%%%%%%%%%%%%%%%%%%%%%%%%%%%%%%%

\subsection{Cases (A) and (B) under the requirement $\kappa=\eta(0)$.}
\label{sec:caseA2}

Within this section we will always assume that $\kappa=\eta(0)$. We start by solving problem \eqref{eq:value} supposing that $\eta$ satisfies the condition of Case (B).
\begin{theorem}
\label{thm:CaseBkappaequal}
Suppose that Case (B) holds. Then Theorem \ref{teo:verify} still holds.
\end{theorem}
\begin{proof}
It suffices to notice that under our assumptions Proposition \ref{prop:b} and Proposition \ref{prop:HJB} still hold. Indeed, if $\kappa=\eta(0)$ one can show by following exactly the same arguments employed in the proof of Proposition \ref{prop:b} that there exists a unique $b^*>\overline{x}$ solving equation \eqref{eqb3}. Moreover, constructing the candidate value function $v$ as in \eqref{candidate}, the proof of Proposition \ref{prop:HJB} still works under the requirement $\kappa=\eta(0)$. In particular, one can still prove that $v'\geq \eta$ on $\mathbb{R}_+$ by following the same rational adopted in \emph{Step 3} of the proof of Proposition \ref{prop:HJB}.
\end{proof}

The next theorem provides the solution to \eqref{eq:value} in Case (A). 
\begin{theorem}
\label{thm:CaseAkappaequal}
Suppose that Case (A) holds. Then the function
\begin{equation}
\label{def:candidatekappaequal}
v(x):= \int_0^x\eta(y)dy + \frac{1}{r}\Big(\frac{1}{2}\sigma^2(0)\eta'(0) + \mu(0)\eta(0)\Big), \qquad x \geq 0,
\end{equation}
is such that $v=V$ on $\mathbb{R}_+$. 

Moreover, for $\delta>0$ let $b^*_{\delta}>0$ be the unique solution to \eqref{eqb3} when $\kappa=\eta(0)+\delta$ (cf.\ Proposition \ref{prop:b}), and denote by $D^{*,\delta}$ the associated optimal control process that makes $X^{D^{*,\delta}}$ reflected at $b^*_{\delta}$ (cf.\ Theorem \ref{teo:verify}). Then one has that the sequence of admissible controls $(D^{*,\delta})_{\delta>0}$ is maximizing; that is, for any $x\geq 0$
\begin{equation}
\label{minimizing}
\lim_{\delta \downarrow 0} \E_x\bigg[\int_0^{\infty}e^{-rt} \eta(X^{D^{*,\delta}}_t) \circ dD^{*,\delta}_t - \eta(0) \int_0^{\infty}e^{-rt}dL^{D^{*,\delta}}_t\bigg] = V(x).
\end{equation} 
\end{theorem}
\begin{proof}
\emph{Step 1.} An integration by parts reveals that for any $x\geq0$
\begin{align*}
%\label{ineq-kappaequal}
& 0 \geq \int_{0}^x \big(\cL_{\X} - (r-\mu'(y))\big)\eta(y) dy \nonumber \\ 
&= \frac{1}{2}\sigma^2(x)\eta'(x) + \mu(x)\eta(x) - r\bigg[\int_{0}^x \eta(y) dy + \frac{1}{r}\Big(\frac{1}{2}\sigma^2(0)\eta'(0) + \mu(0)\eta(0)\Big)\bigg]  \nonumber \\
& = \big(\cL_X - r\big)v(x). \nonumber
\end{align*}
The latter, together with the fact that $v'(x)=\eta(x)$ and $v'(0)=\eta(0)=\kappa$, show that $v$ solves \eqref{HJB} with its associated Neumann boundary condition at zero. Therefore, by proceeding as in \emph{Step 1} of the proof of Theorem \ref{teo:verify} one has that $v\geq V$ on $\mathbb{R}_+$.
\vspace{0.25cm}

\emph{Step 2.} For a given $\delta>0$, let $b^*_{\delta}>0$ be the unique solution to \eqref{eqb3} when $\kappa=\eta(0)+\delta$ (see Proposition \ref{prop:b}). Denoting by 
$$\Theta(b;\delta):= -\delta + \frac{1}{w}\widehat{\phi}(0)\widehat{\psi}(0)\int_0^b \widehat{m}'(y)h(y)\big(\cL_{\X} - (r-\mu'(y))\big)\eta(y) dy,$$
equation \eqref{eqb3} rewrites as $\Theta(b;\delta)=0$, and a simple application of the implicit function theorem yields
$$\frac{\partial}{\partial \delta}b^*_{\delta} = \frac{1}{\frac{\partial}{\partial b}\Theta(b^*_{\delta};\delta)}>0,$$
where the last inequality is due to the fact that 
$$\frac{\partial}{\partial b}\Theta(b;\delta)=\frac{1}{w} \widehat{\phi}(0)\widehat{\psi}(0)\widehat{m}'(b)h(b)\big(\cL_{\X} - (r-\mu'(b))\big)\eta(b)>0$$
because $(\cL_{\X} - (r-\mu'(b)))\eta(b) <0$ for any $b> 0$ and $h(b) < 0$ for all $b>0$.
Then $b^*_0:=\lim_{\delta \downarrow 0}b^*_{\delta}$ exists by monotonicity, and it is not hard to be convinced that $b^*_0=0$.

Recalling \eqref{def:candidatekappaequal}, for any $\delta>0$ and $x\geq b^*_{\delta}$ we can write 
\begin{align}
\label{ineq2kappaequal}
& v(x) - \int_0^{b^*_\delta}\eta(y) dy + \frac{1}{r}\Big(\frac{1}{2}\sigma^2(b^*_{\delta})\eta'(b^*_{\delta}) + \mu(b^*_{\delta})\eta(b^*_{\delta})\Big) - \frac{1}{r}\Big(\frac{1}{2}\sigma^2(0)\eta'(0) + \mu(0)\eta(0)\Big) \nonumber \\
& = \int_{b^*_{\delta}}^x\eta(y)dy + \frac{1}{r}\Big(\frac{1}{2}\sigma^2(b^*_{\delta})\eta'(b^*_{\delta}) + \mu(b^*_{\delta})\eta(b^*_{\delta})\Big)  \\
& = \int_{b^*_{\delta}}^x\eta(y)dy + V_{\delta}(b^*_{\delta}) \leq V(x) \nonumber,
\end{align}
where $V_{\delta}$ denotes the value function \eqref{eq:value} when $\kappa=\eta(0)+\delta$, and $V$ the value function \eqref{eq:value} when $\kappa=\eta(0)$.
The second inequality in \eqref{ineq2kappaequal} is due to \eqref{vb} combined with Theorem \ref{teo:verify}, whereas the third one follows by noticing that
\begin{align*}
& \int_{b^*_{\delta}}^x\eta(y)dy + V_{\delta}(b^*_{\delta}) = \E_x\bigg[\int_{0}^{\infty} e^{-rs}\eta(X^{D^{*,\delta}}_s) \circ dD^{*,\delta}_s - \big(\eta(0)+\delta\big) \int_{0}^{\infty} e^{-rs}dL^{*,\delta}_s\bigg] \nonumber \\
& \leq \E_x\bigg[\int_{0}^{\infty} e^{-rs}\eta(X^{D^{*,\delta}}_s) \circ dD^{*,\delta}_s - \eta(0) \int_{0}^{\infty} e^{-rs}dL^{*,\delta}_s\bigg] \leq V(x).
\end{align*}
Then, taking limits as $\delta \downarrow 0$ in \eqref{ineq2kappaequal} and using that $\lim_{\delta \downarrow 0}b^*_{\delta} =0$ we find $v(x) \leq V(x)$ for any $x\geq0$.
\vspace{0.25cm}

\emph{Step 3.} Combining the results of \emph{Step 1} and \emph{Step 2} we conclude that $v(x) = V(x)$ and that $(D^{*,\delta})_{\delta>0}$ is a maximizing sequence.
\end{proof}

Informally, the previous result shows that if $\kappa=\eta(0)$ and Case (A) holds, then it is optimal to keep the state process at $0$ and to exert control whenever the state process attempts to become strictly positive.

%%%%%%%%%%%%%%%%%%%%%%%%%%%%%%%%%%%%%%%%%%%%%%%%%%%%%%%%%%%%%%%

\subsection{Case (C)}
\label{sec:caseC}

We now assume that Case (C) holds (i.e.\ $(\cL_{\X} - (r - \mu'(x)))\eta(x) \geq 0$ for all $x \geq 0$), and we guess that it is optimal not exerting control at all. This conjecture leads to the candidate optimal control $D^* \equiv 0$ and to the candidate value function
\begin{equation}
\label{eq:candidateB}
v(x) = - \kappa\, \E_x\bigg[\int_0^{\infty} e^{-rt} dL^0_t\bigg], \qquad x \geq 0,
\end{equation}
where $L^0 \in \mathcal{S}$ makes the process $X^0$ reflected at zero. Notice that under Assumption \ref{ass:coef} there indeed exists a (pathwise) unique solution $(X^0,L^0)$ to the Skorokhod reflection problem at zero (cf.\ \cite{Tanaka}, Theorem 4.1). 

\begin{theorem}
\label{thm:verificocaseB}
One has that $v=V$ on $\R_+$ and that $D^* \equiv 0$ is optimal.
\end{theorem}

\begin{proof}
The proof is organized in two steps. We first show that $v$ of \eqref{eq:candidateB} is a classical solution of the HJB equation \eqref{HJB} and it is such that $v'(0)=\kappa$. Then, we prove that $v=V$ on $\R_+$ and that $D^*\equiv 0$ is optimal.
\vspace{0.25cm}

\emph{Step 1.} By Lemma 2.1 and Corollary 2.2 in \cite{Shreveetal} we have that $v$ of \eqref{eq:candidateB} solves the ODE 
\begin{equation}
\label{ODEvB}
\big(\cL_X - r\big)v(x)=0, \qquad x \geq 0,
\end{equation}
and satisfies the boundary conditions $v'(0)=\kappa$ and $\lim_{x \uparrow \infty}v(x)=0$. Hence $v(x) = \kappa \phi(x)/\phi'(0)$, where the fact that $+\infty$ is natural for $X$, and therefore $\lim_{x \uparrow \infty}\psi(x) = \infty$, has been used.

It thus remains to show that $v'(x) \geq \eta(x)$ to conclude that $v$ solves \eqref{HJB}. To this end, we notice that $v'(x) = \kappa \phi'(x)/\phi'(0) = \kappa \widehat{\phi}(x)/\widehat{\phi}(0)$, and therefore it
%given the assumed regularity of $\mu$ and $\sigma$, $v'$ satisfies
%\begin{equation}
%\label{ODEvBprime}
%\big(\cL_{\X} - (r-\mu'(x))\big)v'(x)=0, \qquad x \geq 0,
%\end{equation}
%together with $v'(0)=\kappa$, and therefore 
admits the probabilistic representation
\begin{equation}
\label{vBprimerepr}
v'(x) = \kappa\, \widehat{\E}_x\Big[e^{-\int_0^{\tau_0} (r - \mu'(\X_s))ds}\Big], \qquad x \geq 0,
\end{equation}
where again $\tau_0=\inf\{t\geq 0: \X^x_t \leq 0\}$, $x \geq 0$, and we have used that $\lim_{x \uparrow \infty} \widehat{\phi}(x)=0$ since $+\infty$ is natural for $\X$.

On the other hand, by \eqref{repr-eta} we can also write 
\begin{align}
\label{eta-repr}
& \eta(x) = \widehat{\E}_x\bigg[e^{-\int_0^{\tau_0} (r - \mu'(\X_s))ds}\eta(0) - \int_0^{\tau_0}e^{-\int_0^s (r - \mu'(\X_u))du}\big(\cL_{\X} - (r-\mu'(\X_s))\big)\eta(\X_s) ds\bigg] \\
& \leq  \eta(0)\,\widehat{\E}_x\Big[e^{-\int_0^{\tau_0} (r - \mu'(\X_s))ds}\Big], \nonumber
\end{align}
where the last inequality is due to the fact that $(\cL_{\X} - (r-\mu'(x)))\eta(x) \geq 0$ for all $x \geq 0$. Hence, combining \eqref{vBprimerepr} and \eqref{eta-repr} we find that for any $x\geq 0$
$$v'(x) - \eta(x) \geq \big(\kappa - \eta(0)\big)\, \widehat{\E}_x\Big[e^{-\int_0^{\tau_0} (r - \mu'(\X_s))ds}\Big] \geq 0,$$
since $\kappa \geq \eta(0)$ by assumption, and $\widehat{\P}_x(\tau_0< \infty)>0$ for any $x \in \mathbb{R}$ by regularity of $\X$.

We therefore conclude that $v$ is a classical solution to \eqref{HJB} and it is such that $v'(0)=\kappa$.
\vspace{0.25cm}

\emph{Step 2.} Fix $x\geq 0$ and notice that $D^*\equiv 0$ is such that $D^*\in \mathcal{A}(x)$. Also, since $v$ solves \eqref{HJB} and it is such that $v'(0)=\kappa$, by repeating the arguments employed in \emph{Step 1} of the proof of Theorem \ref{teo:verify} one can show that $v \geq V$ on $\R_+$. However, by definition $v(x) = \mathcal{J}_x(0) \leq V(x)$, and therefore we conclude that $v=V$ on $\R_+$ and that $D^*\equiv 0$ is optimal.
\end{proof}

%%%%%%%%%%%%%%%%%%%%%%%%%%%%%%%%%%%%%%%

\section{A Case Study with Mean-Reverting Dynamics: Sensitivity Analysis}
\label{sec:CS}

In this section we take $\eta(x)=\eta_o>0$ for all $x\geq0$ and 
(cf.\ \eqref{state:X}) 
\begin{equation}
\label{state:X-OU}
dX^{D}_t=(\mu - \theta X^D_t)dt+\sigma dB_t+dL^D_t-dD_t,\qquad X^{D}_{0}=x\geq 0,
\end{equation}
for some $\theta > 0$, $\mu \in \R$ and $\sigma>0$. Hence (cf.\ \eqref{state:X1})
\begin{equation}
\label{state:X1-OU}
dX_t=(\mu - \theta X_t)dt+\sigma dB_t,\qquad X_{0}=x \in \mathbb{R},
\end{equation}
and $\widehat{X}\equiv X$ since $\sigma'\equiv 0$ (cf.\ \eqref{state:XX}).
The parameter $\theta$ is the mean-reversion speed towards the asymptotic mean $\mu/\theta \in \R$. Throughout this section we also assume that $\kappa>\eta_o$.

Under this specification, problem \eqref{eq:value} well models an optimal dividend problem with (compulsory) capital injections (see, e.g., \cite{KulenkoSchmidli}, \cite{LokkaZervos}, \cite{Yang} and Chapter 2.5 in \cite{Schmidli} for a general overview on optimal dividend problems), and with a mean-reverting dynamics for the cash reservoir (see \cite{Cadenillas} for a discussion on empirical and economic motivations for such a kind of dynamics). In this setting $\eta_o$ and $\kappa$ represent the constant transaction/administration costs on dividends and on capital injections.

Because $r-\mu'(x)=r + \theta$, the characteristic equation \eqref{ODE2} reads $\frac{1}{2}\sigma^2 f'' + (\mu-\theta x)f' = (r+\theta)u$, $r > 0$, and it is known that it admits the two linearly independent, positive solutions (cf.~\cite{JYC}, p.\ 280)
\begin{equation}
\label{phi}
\widehat{\phi}(x):=
e^{\frac{\theta(x-\frac{\mu}{\theta})^2}{2\sigma^2}}D_{-\frac{r+\theta}{\theta}}\Big(\frac{(x-\frac{\mu}{\theta})}{\sigma}\sqrt{2\theta}\Big)
\end{equation}
and
\begin{equation}
\label{psi}
\widehat{\psi}(x):=
e^{\frac{\theta(x-\frac{\mu}{\theta})^2}{2\sigma^2}}D_{-\frac{r+\theta}{\theta}}\Big(-\frac{(x-\frac{\mu}{\theta})}{\sigma}\sqrt{2\theta}\Big),
\end{equation}
which are strictly decreasing and strictly increasing, respectively. In both \eqref{phi} and \eqref{psi} $D_{\alpha}$ is the cylinder function of order $\alpha$ given by (see, e.g., \cite{Trascendental}, Chapter VIII, Section 8.3, eq.\ (3) at page 119)
\begin{align}
\label{cylinder}
D_{\alpha}(x):= \frac{e^{-\frac{x^2}{4}}}{\Gamma(-\alpha)}\int_0^{\infty}t^{-\alpha -1} e^{-\frac{t^2}{2} - x t} dt, \quad \text{Re}(\alpha)<0,
\end{align}
where $\Gamma(\,\cdot\,)$ is the Euler's Gamma function.

Since $\eta(x)=\eta_o$ satisfies the requirement of Case (A) and $\kappa>\eta_o$, we know by Theorem \ref{teo:verify} that the optimal control prescribes to reflect the state variable at zero and at $b^*>0$. In particular, the equation for $b^*$ (cf.\ \eqref{eqb3}) reads
\begin{equation}
\label{eqb3-OU}
\kappa = \eta_o\bigg[1 - \frac{1}{w}(r + \theta)\widehat{\phi}(0)\widehat{\psi}(0)\int_0^b \widehat{m}'(y)h(y) dy\bigg],
\end{equation}
where we recall that $h(x)=\frac{\widehat{\phi}(x)}{\widehat{\phi}(0)}-\frac{\widehat{\psi}(x)}{\widehat{\psi}(0)}$, $\widehat{m}'$ is given by \eqref{hatm}, and $w$ is defined through \eqref{Wronskian}. 

By picking for example

\begin{table}[h]
	\centering
	\begin{tabular}   {| l | l | l | l |l|l|}
		\hline
		$\mu$ & $\sigma^2/2$ & $\theta$ & $r$ & $\kappa$ & $\eta_o$\\ \hline
		0.1 & 0.4 & 1 & 0.05 & 1 & 0.5\\
		\hline
	\end{tabular}
	\vspace{0.25cm}
	\caption{Parameters' values used for the determination of $b^*$.}
	\label{table}
\end{table} 

\noindent equation \eqref{eqb3-OU} can be easily solved with Maple yielding $b^*=0.91$ (this value is approximated at the second decimal digit).

In the following we study the sensitivity of the optimal boundary $b^*$ and of the value function $V$ of \eqref{eq:value} with respect to the model's parameters. 

\begin{proposition}
\label{prop:Vmonotone}
Increased uncertainty decreases the value; that is, if $\sigma_1 \geq \sigma_2$, and $V_i$, $i=1,2$, denotes the value function when the underlying volatility coefficient is $\sigma_i$, $i=1,2$, we have $V_1 \leq V_2$.

Also, denoting by $b^*_i$ the optimal boundary for the problem with volatility $\sigma_i$, we have that $b^*_1 \geq b^*_2$ whenever $\sigma_1\geq \sigma_2$.
\end{proposition}
\begin{proof}
In this proof, for $i=1,2$, we denote by $X_i$ the process \eqref{state:X1-OU} when the volatility is $\sigma_i$, and by $\cL_{X_i}$ its associated infinitesimal generator. Also, recall that by Theorem \ref{teo:verify} $V_i$ is given by equation \eqref{candidate} (see also \eqref{vb}).
\vspace{0.25cm}

\emph{Step 1.} We borrow ideas from Theorem 6.1 in \cite{Matomaki} (see also \cite{Alvarez99}) and we prove the first claim. Taking $\sigma_1\geq \sigma_2$, and recalling the expression of $V_i$ (cf.\ \eqref{candidate} and \eqref{vb}) we find by direct calculations
\begin{align}
\label{CS1}
\big(\cL_{X_1} - r\big)V_2(x) =
\left\{
\begin{array}{l}
\displaystyle \frac{1}{2}\big(\sigma^2_1 - \sigma^2_2\big) V''_2(x),\quad 0 \leq x < b^*_2,\\[+6pt]
\displaystyle - \eta_o(r+\theta)(x-b^*_2), \quad x \geq b^*_2
\end{array}
\right.
\end{align}
Clearly $(\cL_{X_1} - r)V_2<0$ on $[b^*_2,\infty)$. Moreover, exploiting the linear structure of \eqref{state:X-OU} and easily adapting arguments from the proof of Lemma 6.1 in \cite{ELKK88}, one can show that $V_2(\,\cdot\,)$ is concave. Therefore $V''_2(x) \leq 0$ on $[0,b^*_2)$, and we thus obtain that $\big(\cL_{X_1} - r)V_2(x) \leq 0$ on $\R_+$. Also, $V'_2(0) = \kappa$ by \eqref{vprimeOS}. By arguing as in \emph{Step 1} of the proof of Theorem \ref{teo:verify} we then find that $V_2 \geq V_1$ on $\R_+$. 
\vspace{0.25cm}

\emph{Step 2.} We now show that $b^*$ is an increasing function of the volatility $\sigma$. To accomplish that we follow a contradiction scheme, and therefore we take $\sigma_1\geq \sigma_2$ and we suppose that $b^*_1 < b^*_2$. We then observe that from \eqref{vb} we have (for $\mu(x)= \mu-\theta x$)
\begin{align}
\label{contradiction}
V_2(b^*_2) & = \frac{1}{r} \eta_o \mu(b^*_2) = \frac{1}{r}\eta_o\big[\mu(b^*_2) - r b^*_2\big] + \eta_o b^*_2 \nonumber \\
& < \frac{1}{r}\eta_o \big[\mu(b^*_1) - r b^*_1\big] + \eta_o b^*_2 = V_1(b^*_1) + \eta_o(b^*_2 -b^*_1) = V_1(b^*_2),
 \end{align}
where the inequality is due to Assumption \ref{ass:rate}. However, equation \eqref{contradiction} contradicts that $V_2 \geq V_1$ on $\R_+$ (cf.\ \emph{Step 1}), and therefore shows that $b^*_1 \geq b^*_2$.
\end{proof}

\begin{proposition}
\label{prop:Vmonotonetheta}
Increased speed of mean reversion decreases the value; that is, if $\theta_1 \geq \theta_2$, and $V_i$, $i=1,2$, denotes the value function when the underlying mean-reversion speed is $\theta_i$, $i=1,2$, we have $V_1 \leq V_2$.
\end{proposition}
\begin{proof}
For $i=1,2$, we denote by $X_i$ the process \eqref{state:X1-OU} when the speed of mean reversion is $\theta_i$, and by $\cL_{X_i}$ its associated infinitesimal generator. Also, recall that by Theorem \ref{teo:verify} $V_i$ is given by equation \eqref{candidate} (see also \eqref{vb}). 

To prove the claim, we now follow arguments analogous to those used in \emph{Step 1} of the proof of Proposition \ref{prop:Vmonotone}. Taking $\theta_1\geq \theta_2$, and recalling the expression of $V_i$ (cf.\ \eqref{candidate} and \eqref{vb}) we find by direct calculations
\begin{align}
\label{CS1-theta}
\big(\cL_{X_1} - r\big)V_2(x) =
\left\{
\begin{array}{l}
\displaystyle \big(\theta_2 - \theta_1\big) xV'_2(x),\quad 0 \leq x < b^*_2,\\[+6pt]
\displaystyle - \eta_o(r+\theta_1)(x-b^*_2), \quad x \geq b^*_2
\end{array}
\right.
\end{align}
Clearly $(\cL_{X_1} - r)V_2<0$ on $[b^*_2,\infty)$. Moreover, since $xV'_2 \geq 0$ by \eqref{vprimeOS} and nonnegativity of $x$, we have that $\big(\cL_{X_1} - r)V_2(x) \leq 0$ on $\R_+$. Also, $V'_2(0)=\kappa$ by \eqref{vprimeOS}. By arguing as in \emph{Step 1} of the proof of Theorem \ref{teo:verify} we then find that $V_2 \geq V_1$ on $\R_+$. 
\end{proof}

Unfortunately, we have not been able to provide an analytical proof of the monotonicity of the free boundary $b^*$ with respect to the speed of mean-reversion $\theta$. However, picking $\mu$, $\sigma^2/2$, $r$, $\kappa$ and $\eta_o$ as in Table \ref{table} we observe a decreasing trend of $b^*$ with respect to $\theta$: 
\begin{table}[h]
	\centering
	\begin{tabular}   {| l | l | l | l | l | l | l | l | l |}
		\hline
		$\theta:$ & 0.25 & 0.50 & 0.75 & 1.0 & 1.25 & 1.50 & 1.75 & 2.0 \\ \hline
		$b^*:$ & 1.61 & 1.22 & 1.03 & 0.91 & 0.82 & 0.75 & 0.70 & 0.66 \\ \hline
	\end{tabular}
	\vspace{0.25cm}
	\caption{Values of $b^*$ when varying $\theta$ (the values of $b^*$ are approximated at the second decimal digit).}
	\label{table2}
\end{table} 

The dependency of $V$ with respect to $\eta_o$ and $\kappa$ is easily obtained from \eqref{eq:functional}: $V$ increases as $\eta_o$ increases, and $V$ decreases as $\kappa$ increases. 

For the proof of the next result an important role is played by the probabilistic representation of $V'$ derived in Proposition \ref{prop:Vprime}. In the following we write $V'(x;\cdot)$ and $b^*(\,\cdot\,)$ in order to stress the dependence of these quantities on a given parameter.

\begin{proposition}
\label{prop:CS}
The following holds true:
\begin{itemize}
\item[(i)] $\kappa \mapsto b^*(\kappa)$ is increasing;
\item[(ii)] $\eta_o \mapsto b^*(\eta_o)$ is increasing.
\end{itemize}
\end{proposition}

\begin{proof}
\emph{(i).} We clearly have from \eqref{vprimeOS} that $V'$ is increasing in $\kappa$. Hence, for $\kappa_2>\kappa_1$ we have
$$b^*(\kappa_1)=\sup\{x \geq 0: V'(x;\kappa_1) > \eta_o\} \leq \sup\{x \geq 0: V'(x;\kappa_2) > \eta_o\} = b^*(\kappa_2).$$
The claim is therefore proved.
\vspace{0.25cm}

\emph{(ii).} The proof employs ideas similar to those used in the proof of (i) above. Indeed, by \eqref{vprimeOS} we have that $V'$ is increasing in $\eta_o$. 
\end{proof}

\section*{Acknowledgments}
\noindent Financial support by the German Research Foundation (DFG) through the Collaborative Research Centre 1283 ``Taming uncertainty and profiting from randomness and low regularity in analysis, stochastics and their applications'' is gratefully acknowledged. I wish to thank Anna Maria Fiori for suggesting reference \cite{Alfarano}, Torben Koch for his help with Maple, and Gabriele Stabile for useful discussions.

%%%%%%%%%%%%%%%%%%%%%%%%%%%%%%%%%%%%%%%%%%%%%%%%%%%%%%%

\appendix

\section{}

\renewcommand{\theequation}{A-\arabic{equation}}

\begin{lemma}
\label{lem:AM}
Under Assumption \ref{ass:psiprimephiprime} one has that $\psi'=\widehat{\psi}$ and $-\phi'=\widehat{\phi}$, where $\widehat{\psi}$ and $\widehat{\phi}$ are the strictly increasing and strictly decreasing fundamental solutions of the ODE $(\cL_{\X} - (r-\mu'))f =0$ for $\X$ killed at rate $r-\mu'$.
\end{lemma}

\begin{proof}
We simply repeat the arguments in the second part of the proof of Lemma 4.3 in \cite{AlvarezMatomaki} (see also Theorem 9 in \cite{Alvarez01}). Under Assumption \ref{ass:coef} standard differentiation reveals that $\psi'$ and $\phi'$ solve the ODE 
\begin{equation}
\label{ODE-AM}
(\cL_{\X} - (r-\mu'))f =0.
\end{equation}
Also, for any $x\in \R$ one has $\phi''(x)\psi'(x) - \phi'(x)\psi''(x)=2rW\S'(x)\neq 0$, and so any solution $f$ to the previous ODE has to be of the form $f(x)= c_1 \psi'(x) + c_2\phi'(x)$. Furthermore, note that under Assumptions \ref{ass:coef} and \ref{ass:rate}, and the supposed boundary behavior of $X$, Corollary 1 of \cite{Alvarez03} can be applied yielding that $\phi$ and $\psi$ are strictly convex.

We thus find that for all $x \in (\ell,r)$ we can write
$$\widehat{\E}_x\Big[e^{-\int_0^{\widehat{\tau}}(r-\mu'(\X_s))ds}\Big] = \frac{f_1(x)}{f_1(\ell)} + \frac{f_2(x)}{f_2(r)},$$
where $\widehat{\tau}:= \inf\{t\geq 0: \X_t \notin (\ell,r)\}$, $\widehat{\P}_x$-a.s., and $f_1(x):=\frac{\phi'(r)}{\psi'(r)}\psi'(x) - \phi'(x)$ and $f_2(x):=\psi'(x) - \frac{\psi'(\ell)}{\phi'(\ell)}\phi'(x)$ are the fundamental increasing and decreasing solutions of \eqref{ODE-AM} when $\X$ is killed at $\ell$ and $r$.

Noticing that $\lim_{\ell \downarrow -\infty}\psi'(\ell)/\phi'(\ell) = 0$ and $\lim_{r \uparrow +\infty}\phi'(r)/\psi'(r) = 0$ by the required boundary behavior of $X$, Assumption \ref{ass:psiprimephiprime} implies that
$$\lim_{\ell \downarrow -\infty}\widehat{\E}_x\Big[e^{-\int_0^{\widehat{\tau}}(r-\mu'(\X_s))ds}\Big] = \frac{\psi'(x)}{\psi'(r)},$$
and
$$\lim_{r \uparrow \infty}\widehat{\E}_x\Big[e^{-\int_0^{\widehat{\tau}}(r-\mu'(\X_s))ds}\Big] = \frac{\phi'(x)}{\phi'(\ell)}.$$
Hence, $\psi'$ and $-\phi'$ are the fundamental solutions of \eqref{ODE-AM} for $\X$ killed at rate $r-\mu'$, and therefore $\psi'=\widehat{\psi}$ and $-\phi'=\widehat{\phi}$.
\end{proof}

\begin{lemma}
\label{Daya}
Let $f \in C^2(\R)$, and for any $x\in \R$ set $\FF(x):=\frac{\widehat{\psi}(x)}{\widehat{\phi}(x)}$ and $y:=\FF(x)$. Then defining
$$\widetilde{f}(y):=\Big(\Big(\frac{f}{\widehat{\phi}}\Big) \circ \FF^{-1}\Big)(y), \qquad y >0,$$
one has that
$$\mbox{ $\widetilde{f}$ is strictly convex at $y>0$\,\, $\Longleftrightarrow$\,\, $(\cL_{\X}-(r-\mu'(x))f(x)>0$ at $x=\FF^{-1}_r(y)$. }$$
\end{lemma}

\begin{proof}
We simply adapt to our setting the calculations indicated in Section 6 of \cite{DayanikKaratzas}. For $y=\FF(x)$ it is obvious that
$\widetilde{f}'(y)=g(x)$ with $g(x):=(f/\widehat{\phi})'(x)/\FF'(x)$ so that $\widetilde{f}''(y)=g'(x)/\FF'(x)$. Since $\FF$ is strictly increasing, we only need to evaluate $g'(x)$. This can be easily done by observing that
\begin{align*}
\FF'(x)=\frac{(\widehat{\psi}'\widehat{\phi}-\widehat{\psi}\widehat{\phi}')(x)}{(\widehat{\phi})^2(x)}=w\frac{\S'(x)}{(\widehat{\phi})^2(x)}\quad\text{and}\quad g(x)=\frac{(f'\widehat{\phi}-f\widehat{\phi}'_r)(x)}{w\,\S'(x)}
\end{align*}
from which we get
\begin{align*}
g'(x)=\frac{\widehat{\phi}(x)(\S' f''-\S'' f')(x)}{w\,(\S')^2(x)}-\frac{f(x)(\S'\widehat{\phi}''- \S''\widehat{\phi}')(x)}{w\,(\S')^2(x)}.
\end{align*}
Now we use that $\S''(x)=-2(\mu(x)+\sigma(x)\sigma'(x))\S'(x)/\sigma^2(x)$ to obtain
\begin{align*}
& g'(x)=\frac{2}{w\,\sigma^2(x)\S'(x)}\Big[\widehat{\phi}(x)\cL_{\X} f(x)-f(x)\cL_{\X} \widehat{\phi}(x) \Big]\nonumber\\
& =\frac{2\widehat{\phi}(x)}{w\,\sigma^2(x)\S'(x)}(\cL_{\X} f-(r-\mu')f)(x),
\end{align*}
where in the last equality we have used that $\cL_{\X}\widehat{\phi}=(r-\mu')\widehat{\phi}$. The last expression proves the claim.
\end{proof}

\begin{lemma}
\label{lem:trasvers}
Let $x\geq 0$, $D \in \mathcal{A}(x)$, and let $(X^{x,D},L^{x,D})$ be the solution to the SDE with reflecting barrier at zero \eqref{state:X}. Then, setting $\tau_n:=\inf\{t\geq 0: X^{x,D}_t \geq n\}$, and recalling $v$ as in \eqref{candidate} one has  
\begin{equation}
\label{trasversality2}
\lim_{n\uparrow \infty}\E_x\big[e^{-r\tau_n}|v(X^{D}_{\tau_n})|\big] = 0.
\end{equation}
\end{lemma}
\begin{proof}
Without loss of generality we may take $n\geq b^*$ so that we can write from \eqref{candidate}
\begin{equation}
\label{eqlemma1}
\E_x\big[e^{-r\tau_n}|v(X^{D}_{\tau_n})|\big] \leq \E_x\big[e^{-r\tau_n}\big]\bigg(v(b^*) + \int_{b^*}^n|\eta(y)|dy\bigg).
\end{equation}
If now 
\begin{equation}
\label{eqlemma2}
\lim_{n\uparrow \infty}\E_x\big[e^{-r\tau_n}\big]\int_{b^*}^n|\eta(y)|dy = 0,
\end{equation}
we conclude that \eqref{trasversality2} holds true, upon recalling that $\tau_n \uparrow \infty$ as $n\uparrow \infty$ and using the dominated convergence theorem.

To prove \eqref{eqlemma2} we preliminary notice that for any $n\geq 1$
\begin{equation}
\label{taun}
\tau_n=\inf\{t\geq 0: X^{x,D}_t \geq n\} \geq \inf\{t\geq 0: X^{x,0}_t \geq n\} =:\sigma_n,
\end{equation}
since $X^{x,D}_t \leq X^{x,0}_t$ a.s.\ for all $t\in[\rho_1,\rho_2)$, for all stopping times $0\leq \rho_1 \leq \rho_2 \leq \infty$. The latter property follows from Corollary 5.5 of \cite{Ma93} with (in the notation of that paper) $x_1=x_2=x\geq0$, $\xi_1 = 0$, $\xi_2 = -D$, and $K_2 = L^D$.

Because $X^{x,0}$ satisfies the reflected SDE
$$dX^{x,0}_t = \mu(X^{x,0}_t)dt + \sigma(X^{x,0}_t)dB_t + dL^0_t, \qquad X^{x,0}_0 = x \geq 0,$$
from Lemma 2.1 and Corollary 2.2 of \cite{Shreveetal} we have that $f(x;n):=\E_x\big[e^{-r\sigma_n}\big]$ solves
\begin{equation}
\label{PDEsigman}
\big(\cL_X - r\big)f(x;n)=0 \quad \text{for}\quad x \in [0,n), \qquad f(n;n)=1, \qquad f'(0;n)=0.
\end{equation}
Hence
\begin{equation}
\label{fn}
f(x;n)=A(n)\psi(x) + B(n)\phi(x), \qquad x \in [0,n],
\end{equation}
where we have set
\begin{equation}
\label{ABlem}
A(n):=-\frac{\phi'(0)}{\phi(n)\psi'(0) - \phi'(0)\psi(n)}, \qquad B(n):=\frac{\psi'(0)}{\phi(n)\psi'(0) - \phi'(0)\psi(n)}.
\end{equation}
Notice that $A(n)> 0$ and $B(n)> 0$ since $\psi$ is strictly increasing and $\phi$ is strictly decreasing.
%However, by the second part of the proof of Lemma 4.3 in \cite{AlvarezMatomaki} (see also Corollary 1 \cite{Alvarez03}), we have that $\widehat{\psi}=\psi'$ and $\widehat{\phi}=-\phi'$, and therefore
%\begin{equation*}
%A(n)=-\frac{\widehat{\phi}(0)}{\phi(n)\widehat{\psi}(0) + \widehat{\phi}(0)\psi(n)}, \qquad B(n)=\frac{\widehat{\psi}(0)}{\phi(n)\widehat{\psi}(0) + \widehat{\phi}(0)\psi(n)}.
%\end{equation*}
The two expressions of \eqref{ABlem} in turn easily yield 
\begin{equation}
\label{ABlem2}
A(n)\leq \frac{1}{\psi(n)}, \qquad B(n)\leq \frac{{\psi'}(0)}{|\phi'(0)|}\frac{1}{\psi(n)}.
\end{equation}
Using \eqref{ABlem2} in \eqref{fn} gives $0 \leq f(x;n) \leq C(x)/\psi(n)$ for some $C(x)>0$, and therefore 
\begin{equation}
\label{final}
0 \leq \E_x\big[e^{-r\tau_n}\big]\int_{b^*}^n|\eta(y)|dy \leq \E_x\big[e^{-r\sigma_n}\big]\int_{b^*}^n|\eta(y)|dy \leq \frac{C(x)}{\psi(n)} \int_{0}^n|\eta(y)|dy.
\end{equation}
Here \eqref{taun} has been employed to write the second inequality.

Since $\lim_{n\uparrow \infty}\psi(n) = + \infty$ as $+\infty$ is natural for $X$, we have two cases: if $\lim_{n\uparrow \infty}\int_{0}^n|\eta(y)|dy < \infty$, we let $n\uparrow \infty$ and we find \eqref{eqlemma2}; if $\lim_{n\uparrow \infty}\int_{0}^n|\eta(y)|dy = +\infty$, we apply De L'H\^opital's rule to obtain 
$$\lim_{n\uparrow \infty} \frac{1}{\psi(n)} \int_{0}^n|\eta(y)|dy = \lim_{n\uparrow \infty} \frac{\eta(n)}{\psi'(n)} = \lim_{n\uparrow \infty} \frac{\eta(n)}{\widehat{\psi}(n)} = 0.$$
Here the penultimate step follows from the fact that $\psi'=\widehat{\psi}$ (cf.\ Lemma \ref{lem:AM}), whereas the last equality is due to Assumption \ref{ass:costs}. We have thus proved \eqref{eqlemma2}, and the proof is therefore completed.
\end{proof}

%%%%%%%%%%%%%%%%%%%%%%%%%%%%%%%%%%%%%%%%%%%%%%%%%%%%

\end{document}